\documentclass[13pt, reqno]{amsart}
\usepackage{dsfont}
\usepackage{amsmath}
\usepackage{amscd}
\usepackage{amsthm}
\usepackage{amssymb} \usepackage{latexsym}
\usepackage{eufrak}
\usepackage{euscript}
\usepackage{epsfig}
\usepackage{graphics}
\usepackage{array}
\usepackage{enumerate}
\usepackage{color}
\usepackage{wasysym}
\usepackage{pdfsync}
\usepackage{stmaryrd}
\usepackage{url}

\newcommand{\bel}[1]{\begin{equation*}\label{#1}}
	
	\newcommand{\be}{\begin{equation}}

		\newcommand{\ba}{\begin{eqnarray}}
			\newcommand{\ea}{\end{eqnarray}}

		\newcommand{\qe}{\end{equation}}
	\newcommand{\R}{{\mathbb R}}
	\newcommand{\N}{{\mathbb N}}
	\newcommand{\Z}{{\mathbb Z}}
	\newcommand{\C}{{\mathbb C}}

	\newcommand{\eg}{\begin{example}}
		\newcommand{\egd}{\end{example}}
	\newcommand{\tm}{\begin{thm}}
		\newcommand{\tmd}{\end{thm}}
	\newcommand{\co}{\begin{coro}}
		\newcommand{\cod}{\end{coro}}
	\newcommand{\enu}{\begin{enumerate}}
		\newcommand{\enud}{\end{enumerate}}
	\newcommand{\rmk}{\begin{rem}}
		\newcommand{\rmkd}{\end{rem}}
	
	\theoremstyle{theorem}
	\newtheorem{thm}{Theorem}[section]
	\newtheorem{prop}[thm]{Proposition}
	\theoremstyle{example}
	\newtheorem{example}[thm]{Example}
	\newtheorem{coro}[thm]{Corollary}
	\theoremstyle{lemma}
	\newtheorem{lemma}[thm]{Lemma}
	\theoremstyle{definition}
	\newtheorem{defi}[thm]{Definition}
	\theoremstyle{proof}
	
	\theoremstyle{remark}
	\newtheorem{rem}[thm]{Remark}
	\theoremstyle{remark}

\begin{document}

		\title[Well-posedness of generalized nonlinear wave equations]{The well-posedness of generalized nonlinear wave equation on the lattice graph}
		
		\author{Bobo Hua}
		\address{Bobo Hua: School of Mathematical Sciences, LMNS,
			Fudan University, Shanghai 200433, China; Shanghai Center for
			Mathematical Sciences, Fudan University, Shanghai 200433,
			China.}
		\email{bobohua@fudan.edu.cn}
		
		\author{Jiajun Wang}
		\address{Jiajun Wang: School of Mathematical Sciences,
			Fudan University, Shanghai 200433, China.}
		\email{21300180146@m.fudan.edu.cn}


		\begin{abstract}

			In this paper, we introduce a novel first-order derivative for functions on a lattice graph, and establish its weak $(1,1)$ estimate as well as strong $(p,p)$ estimate for 
			$p>1$ in weighted spaces. This derivative is designed to reconstruct the discrete Laplacian, enabling an extension of the theory of nonlinear wave equations, including quasilinear wave equations, to lattice graphs. We prove the local well-posedness of generalized quasilinear wave equations and the long-time well-posedness of these equations for small initial data. Furthermore, we prove the global well-posedness of defocusing semilinear wave equations for large initial data.

		\end{abstract}

		\maketitle
		\numberwithin{equation}{section}
		\section{introduction}\label{sec-intro}
		
		Partial differential equations (PDEs) are important mathematical tools. In recent years, there has been growing interest in the study of PDEs in discrete settings, attracting attention from both physics and mathematics. In physics, discrete analogs like the heat equation provide insights into thermal resistance between layers and their dispersion properties \cite{3}. Similarly, discrete wave equations offer descriptions of atomic vibrations in crystalline semiconductors \cite{4}, while discrete Schr\"odinger equations serve as standard models for dynamic media dynamics \cite{7}. Meanwhile, mathematicians are increasingly engaged in the study of discrete PDEs as well.  For instance, Grigor'yan, Lin, and Yang used variational methods to prove the existence of solutions for discrete nonlinear elliptic equations \cite{GLY16a,GLY16b}. Chow and Luo introduced the discrete Ricci flow and gave an alternative proof of the circle packing theorem \cite{25}. There are many interesting results in the literature; see e.g. \cite{Woess00,Barlow17,Grigoryan18,KLW21}.
		
		In this paper, we study discrete nonlinear wave equations on lattice graphs.
		Let $G=(V,E)$ be a locally finite, simple and undirected graph with the set of vertices $V$ and the set of edges $E.$ Two vertices $x,y$ are called neighbours, denoted by $x\sim y,$ if there is an edge connecting them. The discrete Laplacian is defined as, for any $f:V\to\C,$
		$$\Delta f(x)=\sum_{y\in V: y\sim x} f(y)-f(x),\quad x\in V.$$ This definition is motivated from the theories of numerical computation, electric network and random walk etc. \cite{1,2,24}. The $d$-dimensional lattice graph for $d\in \N$ consists of the set of vertices $$\Z^{d}=\lbrace m=(m_{1},\cdots,m_{d}); m_{j}\in \Z, j=1,\cdots, d\rbrace$$ and the set of edges $\{m\sim n; |m-n|=1,m,n\in \Z^d\}.$ For simplicity, we write $\Z^d$ for the $d$-dimensional lattice graph. We denote by $C_{0}(\Z^{d})$ the set of finitely supported functions on $\Z^d.$

		The discrete linear wave equation on a graph $G$ is defined as
		$$\partial_{t}^2 u(x,t) - \Delta u(x,t)=0,\quad u\in C^2_t(V\times [0,T]).$$
		Friedman and Tillich proved that the property of finite propagation speed fails for discrete linear wave equation \cite{8}. Han and the first author constructed a nontrivial solution to the Cauchy problem on $\Z$ for discrete linear wave equation with zero initial data \cite{9}, which is an analog of the Tychnoff solution to the heat equation. 
		Schultz derived dispersive estimates for discrete linear wave equations on lattice graphs $\Z^2$ and $\Z^3$ \cite{5}, and proved the existence of corresponding semilinear wave equations $\partial_{t}^2 u(x,t) - \Delta u(x,t) = F(u)$ with proper growth for the nonlinearity $F(u).$ Schultz's results were further extended to $\Z^4$ and $\Z^5$ \cite{BCH23,27}.
		See \cite{MaWang12,LX19,LX22,Hong23} for other results on discrete wave equations.	
		
		The aim of the paper is to formulate quasilinear wave equations on lattice graphs, and prove well-posedness results for them. Let $D_{j}$ be the difference operator on $\Z^d$ defined as $$D_j u(m)=u(m+e_j)-u(m),\quad u:\Z^d\to \C, m\in \Z^d,$$ where $\lbrace e_{j}\rbrace_{j=1}^{d}$ is the standard coordinate basis of $\Z^{d}.$	Direct computation shows that $$\Delta u\neq \sum_j D_j\circ D_j u,$$ which poses an obstacle on extending classical theory of nonlinear wave equations to graphs. To circumvent the difficulty, we introduce a new definition of discrete partial derivative, which is compatible with the discrete Laplacian.		\begin{defi}\label{discrete partial}
			Discrete partial derivative $\partial_{j}u$ for $u\in \bigcup_{0<p\le \infty}\ell^{p}(\Z^{d})$  is defined by convolution operator $u\ast \varphi_{j}$, where $\varphi_{j}$ is given by 
			\begin{equation*}\label{def}
				\varphi_{j}(m) :=\left\{
				\begin{array}{ll}
					\frac{-4i}{\pi (4a^{2}-1)}, \ & \mathrm{if}\  m=a e_j, a\in \Z,\\
					0,\ &\mathrm{otherwise.}
				\end{array}\right.
			\end{equation*}
		\end{defi}
		To explain the motivation of the above definition, we recall the discrete Fourier transform $\mathcal{F}$ on $\Z^d.$ We denote by $\mathbb{T}^{d}$ the $d$-dimensional torus parametrized by $[-\pi,\pi)^d.$
		\begin{defi}
			For $u\in \ell^1(\Z^d)$ and $g\in L^{1}(\mathbb{T}^{d})$, the discrete Fourier transform $\mathcal{F}$ and inverse discrete Fourier transform $\mathcal{F}^{-1}$ are defined as
			\begin{equation*}
				\mathcal{F}(u)(x):=\sum_{k\in \Z^{d}}u(k)e^{-ikx}, \quad \forall x\in \mathbb{T}^{d},
			\end{equation*}
			\begin{equation*}
				\mathcal{F}^{-1}(g)(k):=\frac{1}{(2\pi)^{d}}\int_{\mathbb{T}^d} g(x)e^{ikx}dx, \quad \forall k\in \Z^{d}.
			\end{equation*}
		\end{defi} 
		
		In fact, the discrete Fourier transform can be extended to $\ell^{p}(\Z^{d})$ for $p\in[1,2],$ which is an isometric isomorphism between $\ell^{2}(\Z^{d})$ and $L^{2}(\mathbb{T}^{d})$. Moreover, one can further extend its definition to more general spaces.
		
		\begin{rem}	
			\begin{enumerate}[(i)] \item	Using the discrete Fourier transform, one can show that for any $u\in C_{0}(\Z^{d}),$
				\begin{equation*}
					\partial_{j}u= \mathcal{F}^{-1}\left(2i\cdot \sin(\frac{x_{j}}{2})\mathcal{F}(u)(x)\right),     j=1,\cdots ,d
				\end{equation*} 
				where $x=(x_{1},\cdots ,x_{d})\in \mathbb{T}^d $ and $i=\sqrt{-1},$ see Proposition~\ref{dispartial1}.
				Since the Fourier multiplier of discrete Laplacian $\Delta$ is $-4\sum_{j=1}^{d}\sin^{2}(\frac{x_{j}}{2}),$ we have $\Delta=\sum_{j=1}^{d}\partial_{j}\circ\partial_{j}.$ This definition aligns with the classical property that the standard difference 
				$D_ju$ does not satisfy.
				\item By Definition~\ref{discrete partial}, $\partial_j u$ is a nonlocal operator, for which $\partial_ju(m)$ depends on all vertices $m+ke_j,$ $k\in \Z.$ For the case of $\Z,$ one can show that $\partial_1 u=i\sqrt{-\Delta} u,$ which coincides with a well-known fractional Laplacian that has been extensively studied in the literature, see \cite{28,Ciaurri18,Ciaurri18b,KellerN23,Wang23}.
			\end{enumerate}
		\end{rem} 
		
		For the discrete $H^1$ space, $H^1(\Z^d):=\{u\in \ell^2(\Z^d); D_j u\in \ell^2(\Z^d), \forall 1\leq j\leq d\},$ it is well-known that by the triangle inequality $H^1(\Z^d)=\ell^2(\Z^d),$ which indicates that it doesn't possess higher regularity. In order to present well-posedness results, we introduce weighted $\ell^{p}(\Z^{d})$ spaces. 
		\begin{defi}
			For $0<p\le \infty$, $\alpha\in \R$, 
			the $\ell^{p,\alpha}$ norm is defined as
			$$\|f\|_{\ell^{p,\alpha}}:=\|f_\alpha\|_{\ell^p(\Z^d)},$$ where $f_\alpha(m):=f(m)\langle m \rangle^\alpha, m\in\Z^d.$ We write
			\begin{equation*}
				\ell^{p,\alpha}(\Z^{d}):=\lbrace f:\Z^d\to \C; \|f\|_{\ell^{p,\alpha}}<\infty \rbrace.
			\end{equation*}
		\end{defi} In fact, we prove that $u\in \ell^{2,1}(\Z^{d})$ if and only if $\mathcal{F}(u)\in H^1(\mathbb{T}^d),$ see Theorem~\ref{iso}. As a subspace of $\ell^2(\Z^d),$ $\ell^{2,1}(\Z^{d})$ consists of functions decaying at least $o(\langle m \rangle^{-1})$ at infinity.
		
		Inspired by results in \cite{10,11,12,13}, we prove the well-posedness theory of generalized discrete nonlinear wave equations, including quasilinear wave equations, in the framework of $\ell^{2,k}(\Z^{d})$, for $k=0,1$. In the following, we follow Einstein's summation convention and write $\partial:=(\partial_{1}, \cdots, \partial_{d} )$, $\partial_{jk}:=\partial_{j}\circ \partial_{k}$,  $u':=(\partial_{t},\partial).$

		\par 
		\begin{thm}\label{Th1} For the following equation
			\begin{equation}\label{equation1}
				\left\{
				\begin{aligned}
					& \partial_{t}^2 u(x,t) - g^{jk}(u,u')\partial_{jk} u(x,t) = F(u,u'), \\
					& u(x,0) = f(x),\quad\partial_t u(x,0) = g(x),\quad (x,t)\in\Z^d\times \R,
				\end{aligned}
				\right.
			\end{equation}
			if $g^{jk}\in C(\mathbb{R}\times \mathbb{R}^{d+1})$, $F \in C^{1}(\mathbb{R}\times \mathbb{R}^{d+1})$, $F(0,0)=0$ and $f,g\in \ell^{2,k}(\Z^{d})$ for $k=0$ or $1$, then it has a unique classical solution $u\in C^{2}(\left[0,T\right];\ell^{2,k}(\Z^{d}))$ for some $T>0.$ Moreover, we have the continuation criterion: if maximal existence time $T^{*}$ is finite, then $\|u(\cdot,t)\|_{\ell^{\infty}(\Z^{d})}+\|\partial_{t}u(\cdot,t)\|_{\ell^{\infty}(\Z^{d})}$
			is unbounded in $\left[0,T^{*}\right)$.
		\end{thm}
		Compared with the ill-posedness of classical framework $C_{t}^{2}(\Z^{d}\times[0,T]),$ see \cite{9}, we prove the local existence and uniqueness of the solution and derive a continuation criterion for global existence and uniqueness in the framework of $C^{2}(\left[0,T\right];\ell^{2,k}(\Z^{d}))$ for the quasilinear wave equation.
		
		\par 
		The next result is the long-time well-posedness of the following discrete quasilinear wave equation with small initial data.
		\begin{thm}\label{Th2} For the equation (\ref{equation2}) with same hypothesis as in Theorem \ref{Th1},
			\begin{equation}\label{equation2}
				\left\{
				\begin{aligned}
					& \partial_{t}^2 u(x,t) - g^{jk}(u, u')\partial_{jk} u(x,t) = F(u,u'), \\
					& u(x,0) = \varepsilon f(x),\quad\partial_t u(x,0) =\varepsilon g(x),\quad (x,t)\in\Z^d\times \R,
				\end{aligned}
				\right.
			\end{equation} 
			there exists $\delta>0$, such that the maximal existence time $$T^{\ast}\ge K\cdot \log(\log(\frac{1}{\varepsilon})),\quad \forall  0<\varepsilon<\delta,$$ where $K=K(F,g^{jk},f,g,d)$ is a positive constant.
		\end{thm}
		\par Based on the continuation criterion in Theorem \ref{Th1} and energy conservation established in Section 3, we prove the global well-posedness of the defocusing discrete nonlinear wave equation with large data $f,g\in \ell^{2,k}(\Z^{d}), k=0,1$.
		\begin{thm}\label{Th3} For $f,g\in \ell^{2,k}(\Z^{d})$ for $k=0$ or $1,$ the equation
			\begin{equation}\label{equation3}
				\left\{
				\begin{aligned}
					& \partial_{t}^2 u(x,t) - \Delta u(x,t) = -|u|^{p-1}u, \\
					& u(x,0) = f(x),\quad\partial_t u(x,0) = g(x),\quad (x,t)\in\Z^d\times \R
				\end{aligned}
				\right.
			\end{equation} 
			has a global and unique classical solution $u\in C^{2}(\left[0,T\right];\ell^{2,k}(\Z^{d}))$.
		\end{thm}

		\par 
		We organize this paper as follows. In Section 2, we prove some useful properties for discrete partial derivative $\partial_{j},$ see Theorem~\ref{bdn-p} for its $(p,p)$ boundedness in weighted spaces. In Section 3, we derive energy estimates for discrete wave equations, which are key properties for the proof of the well-posedness for nonlinear wave equations. In Section 4 and Section 5, we establish the well-posedness theory for nonlinear wave equations. Finally, some interesting and useful results are collected in the last section.

		\noindent
		\textbf{Notation.}
		\begin{itemize}
			\item  By $u\in C_{t}^{k}(\Z^{d}\times [0,T])$, we mean $u(x,t)$ is $C^k$ continuous in time for any fixed vertex $x.$
			\item By $u\in C^{k}([0,T]; B)( or \; L^{p}([0,T];B))$ for a Banach space $B,$ we mean $u$ is a $C^{k}(or \; L^{p})$ map from $[0,T]$ to $B;$ see e.g. \cite{14}.
			\item By $u\in C_{0}(\Z^{d}\times [0,T])$, we mean $u$ has compact support on $\Z^{d}\times [0,T]$.
			\item By $A\lesssim B$ (resp. $A\approx B$), we mean there is a positive constant $C$, such that $A\le CB$ (resp. $C^{-1}B\le A \le C B$). If the constant $C$ depends on $p,$ then we write $A\lesssim_{p}B$ (resp. $A\approx_{p} B$).
			\item Set $\langle m \rangle:=(1+|m|^{2})^{1/2}$ and $|m|:=(\sum_{j=1}^{d}|m_{j}|^{2})^{1/2}$ for $m=(m_{1},\cdots, m_{d})\in \Z^{d}$.
			
		\end{itemize}
		
		\section{Operator properties for discrete partial derivatives}
		In this section, we prove some properties for the discrete partial derivative defined in the introduction.
		\begin{prop}\label{dispartial1} For any $u\in C_0(\Z^d),$
			$$\partial_{j}u= \mathcal{F}^{-1}\left(2i\cdot \sin(\frac{x_{j}}{2})\mathcal{F}(u)(x)\right),     j=1,\cdots ,d.$$
		\end{prop}
		\begin{proof} Since the general case is similar, we only prove the case for $d=1$. 
			Note that \begin{equation*}
				\dfrac{1}{2\pi}\int_{0}^{2\pi} \mathcal{F}(u)(x)2i\cdot \sin\left(\dfrac{x}{2}\right)e^{ikx}dx=\dfrac{1}{2\pi}\int_{0}^{2\pi} \sum_{m\in \Z}u(m) 2i\cdot \sin\left(\dfrac{x}{2}\right)e^{i(k-m)x}dx.
			\end{equation*} 
			By calculation, the imaginary part of the above equals to
			\begin{equation*}
				i\sum_{m\in \Z}\int_{0}^{2\pi}\frac{2\sin(\frac{x}{2})\cdot \cos(k-m)x}{2\pi} u(m)dx =\sum_{m\in \Z}\frac{-4i}{\pi[4(k-m)^{2}-1]}\cdot u(m).
			\end{equation*} 
			And its real part is zero.
			This proves the result.
		\end{proof}

		We recall well-known Young's inequality \cite{15}.
		\begin{lemma}\label{Young}
			Let $\lambda$ be the left Haar measure on a locally compact group G, that satisfies $\lambda(A)=\lambda(A^{-1})$, for any measurable A$\subseteq$G , $A^{-1}:=\lbrace g^{-1}; g\in A\rbrace$. Let $1\le p,q,r\le \infty$ satisfy
			\begin{equation*}
				\frac{1}{q}+1=\dfrac{1}{p}+\dfrac{1}{r}.
			\end{equation*}
			Then for every $f \in L^{p}(G)$ and $g\in L^{r}(G)$ we have the following inequality
			\begin{equation*}
				\|f\ast g\|_{L^{q}(G)} \lesssim_{p,q,r}\| g\|_{L^{r}(G)} \cdot \|f\|_{L^{p}(G)},
			\end{equation*}
			where the convolution $f\ast g$ is defined as 
			\begin{equation*}
				(f\ast g)(x)=\int_{G}f(y)g(y^{-1}x)d\lambda(y).
			\end{equation*}
		\end{lemma}
		\begin{rem}
			In this paper, we will apply Young's inequality to the special case $G=\Z^{d}$ and $\lambda$ is the counting measure. In this case, the convolution is given by $(f\ast g)(k)=\sum_{m\in \Z^{d}}f(k-m)g(m)$. In fact, Young's inequality can be strengthened by replacing $\| \cdot\|_{L^{r}(G)}$ with weaker Lorentz norm $\| \cdot\|_{L^{r,\infty}(G)}$ and requiring $1< p,q,r<\infty$ instead of $1\le p,q,r\le \infty$, but we don't need this stronger version in developing operator properties for discrete partial derivatives.
		\end{rem}
		
		\par 
		The $\ell^{p}$-boundedness of the discrete partial derivative is straightforward.
		\begin{thm} For $f\in \ell^{p}(\Z^{d})$, $1 \le p \le \infty$, $j=1, \cdots , d$, we have			\begin{equation*}
				\|\partial_{j}f \|_{\ell^{p}(\Z^{d})}\lesssim_{p} \| f\|_{\ell^{p}(\Z^{d})}.
			\end{equation*}
			
		\end{thm}
		\begin{proof}
			From Definition~\ref{discrete partial},  $\partial_{j}u$ can be written as $u\ast \varphi_{j}$. A simple observation shows that $\varphi_{j}\in \ell^{1}(\Z^{d}).$ Hence, we can apply Lemma~\ref{Young} and immediately get the result.
		\end{proof}
		Another interesting property is that the operator norm of discrete partial derivative $\partial_{j}$ and that of difference operator $D_{j}$ are equivalent. We first recall the discrete Sobolev seminorm \cite{16}.
		\begin{defi}
			Let $G=(V,E)$ be a locally finite graph. For $f:V\to \C$ and $1\le p\le\infty,$the discrete Sobolev seminorm $\|\cdot\|_{D^{1,p}(G)}$  is defined as
			\begin{equation*}
				||f||_{D^{1,p}(G)} :=\left\{
				\begin{aligned}
					&\left(\frac12\sum_{u,v\in V:u\sim v} |f(u)-f(v)|^p\right)^{\frac{1}{p}},\ p\in[1,\infty),\\
					&\sup_{u,v\in V:u\sim v} |f(u)-f(v)|,\ p=\infty.
				\end{aligned}\right.
			\end{equation*}
		\end{defi}
		Now we have the following result.		\begin{thm}
			For $1<p<\infty$ and $f:\Z^{d}\to \C,$
			\begin{equation*}
				\|\partial_{j}f\|_{\ell^{p}(\Z^{d})}\approx \|D_{j}f\|_{\ell^{p}(\Z^{d})},\quad	\|\partial f\|_{\ell^{p}(\Z^{d})}\approx \|f\|_{D^{1,p}(\Z^{d})}.
			\end{equation*} 
		\end{thm}
		\begin{proof}
			The first equivalence obviously implies the second, hence we only need to prove the former. Note that the discrete partial derivative $\partial_{j}$ and difference operator $D_{j}$ of a function at the vertex $k$ only involve function values on vertices $k+a e_j, a\in \Z.$ Hence, it suffices to prove the result for the case $d=1$. A key observation is that $\partial_1 u=i\sqrt{-\Delta} u$ for $\Z.$ Therefore, applying the boundedness of the Riesz transform in \cite{22,23}, we prove the result.
		\end{proof}
		
		\par 
		Our main results on discrete wave equations are based on the $\ell^{2}(\Z^{d})$ space and generalized $\ell^{2,1}(\Z^{d})$ space. Key properties are that both of them are Hilbert spaces and are isometrically isomorphic to $L^2(\mathbb{T}^d)$ and $H^{1}(\mathbb{T}^d)$ respectively by the Fourier transform; see Theorem \ref{iso}. Besides, in inner product spaces we have following useful integration by part formula.		
		\begin{thm}\label{integration by part}
			Suppose $u,v\in \ell^{2}(\Z^{d})$, $j=1,\cdots,d$, then we have 			\begin{equation*}
				\sum_{k\in\Z^{d}}\partial_{j}u(k)\overline{v(k)}=-\sum_{k\in\Z^{d}}u(k)\overline{\partial_{j}v(k)}.
			\end{equation*}
			If $u,v$ are real-valued functions, then we have
			\begin{equation*}
				\sum_{k\in\Z^{d}}\partial_{j}u(k)v(k)=\sum_{k\in\Z^{d}}u(k)\partial_{j}v(k).
			\end{equation*}
			
		\end{thm}	
		\begin{proof}
			Since the discrete Fourier transform is an isometric isomorphism between $\ell^{2}(\Z^{d})$ and $L^{2}(\mathbb{T}^d)$, we have the following identity
			\begin{equation*}
				RHS=\langle	\mathcal{F}(u),\mathcal{F}(-\partial_{j}v) \rangle_{L^{2}(\mathbb{T}^d)}=\frac{-1}{(2\pi)^{d}}\int_{\mathbb{T}^d}\mathcal{F}(u)(x)\cdot \overline{2i \sin\left(\dfrac{x}{2}\right)\mathcal{F}(v)(x)}dx
			\end{equation*} 
			\begin{equation*}
				=\frac{1}{(2\pi)^{d}}\int_{\mathbb{T}^d}2i \sin\left(\frac{x}{2}\right)\mathcal{F}(u)(x)\overline{\mathcal{F}(v)(x)}=\langle	\mathcal{F}(\partial_{j}u),\mathcal{F}(v) \rangle_{L^{2}(\mathbb{T}^d)}=LHS. 
			\end{equation*}
			Recalling the calculation in Proposition~\ref{dispartial1}, we know that the discrete partial derivative of a real-valued function is purely imaginary, which implies the second assertion.
		\end{proof}
		From the above proof, the discrete partial derivative is in fact a skew-adjoint operator. Besides, the importance of this integration by part formula lies in the loss of derivation rule of the product. More precisely, there is no similar rule such as $\partial_{j}(uv)=\partial_{j}(v)u+\partial_{j}(u)v$, which poses an obstacle on transferring the derivative between two functions. Fortunately, Theorem~\ref{integration by part} provides such a useful tool in the global sense.
		Next, we give some remarks on $\ell^{p,\alpha}(\Z^{d})$.
		\begin{rem}
			By H$\ddot{o}$lder's inequality, the $\ell^{p,\alpha}(\Z^{d})$-norm$(\alpha>0)$ is stronger than the traditional $\ell^{p}(\Z^{d})$-norm. Note that 
			\begin{equation*}
				\sum_{m\in \Z^{d}}|f(m)|^{q}\le \left(\sum_{m\in \Z^{d}}|f(m)|^{p} \langle m \rangle^{\alpha p}\right)^{q/p}\left(\sum_{m\in\Z^{d}}\langle m \rangle^{-\frac{\alpha pq}{p-q}}\right)^{1-q/p}.
			\end{equation*}
			Therefore, if $\dfrac{\alpha pq}{p-q}>d$, then we obtain that $\ell^{p,\alpha}(\Z^{d})\subseteq\ell^{q}(\Z^{d})$, for any $\dfrac{dp}{d+\alpha p}<q\le p$. Since $\ell^{p}(\Z^{d})\subseteq \ell^{r}(\Z^{d}), \forall r>p$, we further deduce that $\ell^{p,\alpha}(\Z^{d})\subseteq\ell^{q}(\Z^{d})$, for any $\dfrac{dp}{d+\alpha p}<q\le \infty$. In particular, we have $\ell^{p,\alpha}(\Z^{d})\subseteq\ell^{p}(\Z^{d})$.
		\end{rem}
		\begin{rem}
			There are two main reasons to define $\ell^{p,\alpha}(\Z^{d})$-norm. The first reason is that the traditional $\ell^{p}(\Z^{d})$-norm has no kind of regularity in the discrete setting, while we need some stronger norm to inherit Sobolev type estimate from classical theory of nonlinear wave equation. The second reason is from the following observation. In Sobolev type estimate, we use the derivative of a function to control itself, but in discrete setting the situation is reversed. Hence, in the traditional framework there is no analogue of Sobolev spaces. However, based on the connection between $\Z^{d}$ and $\mathbb{T}^d$, the classical Sobolev structure in  $\mathbb{T}^d$ maybe a candidate for the Sobolev structure in $\Z^{d}$. We prove the following isomorphism.
		\end{rem}
		\begin{thm}\label{iso}
			$\mathcal{F}^{-1}: H^{k}(\mathbb{T}^{d}) \to \ell^{2,k}(\Z^{d})$ is an isomorphism, with the equivalence $\|\mathcal{F}^{-1}(f)\|_{\ell^{2,k}(\Z^{d})}\approx \|f\|_{H^{k}(\mathbb{T}^{d})}$, $\forall f\in H^{k}(\mathbb{T}^{d}), k\ge0.$
		\end{thm}
		\begin{proof}
			For simplicity, we only deal with the case $k=1$ and $d=1$, since the general case is similar. For $f\in H^{1}(\mathbb{T}^{d})$, we have $f'\in L^{2}(\mathbb{T}^{d}).$ Then 
			\begin{equation*}
				\mathcal{F}^{-1}(f')(m)=\frac{1}{2\pi}\int_{0}^{2\pi}f'(x)e^{imx}dx=-im\cdot \mathcal{F}^{-1}(f)(m).
			\end{equation*}
			As $\mathcal{F}$ is an isometric isomorphism between $\ell^{2}(\Z)$ and $L^{2}(\mathbb{T})$, we have
			\begin{equation*}
				\|f\|_{H^{1}(\mathbb{T})}=\|f\|_{L^{2}(\mathbb{T})}+\|f'\|_{L^{2}(\mathbb{T})}=\left(\sum_{m\in \Z}|\mathcal{F}^{-1}(f)(m)|^2\right)^{1/2}+\left(\sum_{m\in \Z}|-im \mathcal{F}^{-1}(f)(m)|^2\right)^{1/2}
			\end{equation*}
			\begin{equation*}
				\approx {\left(\sum_{m\in \Z}|\langle m \rangle \mathcal{F}^{-1}(f)(m)|^2\right)^{1/2}}=\|\mathcal{F}^{-1}(f)\|_{\ell^{2,1}(\Z)}. 
			\end{equation*}
			Similarly, we can prove that $\mathcal{F}:\ell^{2,k}(\Z)\to H^{k}(\mathbb{T})$ is a bi-Lipschitz equivalence $\|u\|_{\ell^{2,k}(\Z)}\approx \|\mathcal{F}(u)\|_{H^{k}(\mathbb{T})}$,$\forall u\in \ell^{2,k}(\Z) $. Thus, $\mathcal{F}^{-1}$ is an isomorphism. 
		\end{proof}

		Now, we derive some important properties of discrete partial derivatives on the $\ell^{p,\alpha}(\Z^{d})$-norm, which will be useful in the study of discrete nonlinear wave equations in the framework of $\ell^{p,\alpha}(\Z^{d})$-space.
		\par 
		First, we notice that there is no satisfactory $\ell^{p,\alpha}$-boundedness theory for the discrete partial derivative when $\alpha >1$, in the sense that, there exist $p>1$, such that for $j=1,\cdots,d$, $\|\partial_{j}f\|_{\ell^{p,\alpha}(\Z^{d})}\lesssim \|f\|_{\ell^{p,\alpha}(\Z^{d})}$ fails. In fact, we can take $f=\delta_{0}$ (whose value is $1$ at $k=0$ and $0$ elsewhere) and, without loss of generality, we still assume $d=1$. Hence,
		\begin{equation*}
			(\partial_1 \delta_{0})(m)=\sum_{k\in \Z}\delta_{0}(k)\varphi(m-k)=\varphi(m)=\frac{-4i}{\pi(4m^2-1)},
		\end{equation*}
		which  does not belong to $\ell^{p,\alpha}(\Z)$, when $p\le\dfrac{1}{2-\alpha}(>1)$.
		\par 
		Therefore, the only interesting case left is $\alpha=1$. We briefly recall weak $L^{p}$ spaces and the famous Marcinkiewicz Interpolation Theorem \cite{15}.
		\begin{defi}
			For $0<p<\infty$, let $(X,\mu)$ be the measure space. The so-called weak $L^{p}(X,\mu)$, denoted by $L^{p,\infty}(X,\mu),$ is defined as the set of all measurable functions $f$ satisfying
			\begin{equation*}
				\|f\|_{L^{p,\infty}(X,\mu)}:=\sup_{\alpha >0}\lbrace \alpha d_{f}(\alpha)^{1/p} \rbrace<\infty,
			\end{equation*}
			where $d_{f}(\alpha):=\mu(\lbrace |f|\ge \alpha \rbrace)$ is the distribution function of $f$. Additionally, the weak $L^{\infty}(X,\mu)$ is defined as the original $L^{\infty}(X,\mu)$ space.
		\end{defi}
		\begin{rem}
			By definition and Chebyshev's inequality,  one sees that $L^{p}(X,\mu) \subseteq L^{p,\infty}(X,\mu) $ and $\|\cdot\|_{L^{p,\infty}(X,\mu)}\le \|\cdot\|_{L^{p}(X,\mu)} $. When $X=\Z^{d}$ and $\mu$ is counting measure, the weak $L^{p}(X,\mu)$ norm is just $\sup_{\alpha >0}\lbrace \alpha |\lbrace |f|\ge\ \alpha\rbrace|^{1/p}\rbrace$, where $|A|$ means the cardinality of a set A.
		\end{rem}
		Next, we recall the useful Marcinkiewicz Interpolation Theorem.
		\begin{lemma}\label{interpolation}
			Let $(X,\mu)$ and $(Y,\nu)$ be measure spaces and $0<p<q\le \infty,$ and $T$ be a linear operator defined on the space $L^{p}(X)+L^{q}(X)$ with values in the space of measurable functions on $Y$. Suppose that there are two endpoint weak-boundedness as follows
			\begin{equation*}
				\|T(f)\|_{L^{p,\infty}(Y)}\lesssim \|f\|_{L^{p}(X)}  , \forall f\in L^{p}(X),
			\end{equation*}
			\begin{equation*}
				\|T(f)\|_{L^{q,\infty}(Y)}\lesssim \|f\|_{L^{q}(X)}  , \forall f\in L^{q}(X).
			\end{equation*}
			Then for all $p<r<q$ and for all $f\in L^{r}(X)$ we have the $L^{r}$-boundedness
			\begin{equation*}
				\|T(f)\|_{L^{r}(Y)}\lesssim\|f\|_{L^{r}(X)}  , \forall f\in L^{r}(X).
			\end{equation*}
		\end{lemma}
		Now we state our key result for operator properties of discrete partial derivatives. These are the weak $(1,1)$ boundedness and strong $(p,p)$ boundedness in weighted spaces.
		\begin{thm}\label{bdn-p}
			For any $1<p\le \infty$, we have the $\ell^{p,1}$-boundedness for the discrete partial derivative $\partial_j,$ $j=1,\cdots,d$, that is, for any $f\in \ell^{p,1}(\Z^{d})$,
			\begin{equation*}
				\|\partial_{j}f\|_{\ell^{p,1}(\Z^{d})}\lesssim_{p}\|f\|_{\ell^{p,1}(\Z^{d})}.
			\end{equation*}
			For the critical case $p=1$, the $\ell^{1,1}$-boundedness of discrete partial derivative fails, but the weak $\ell^{1,1}$-boundedness holds, that is, for any $f\in \ell^{1,1}(\Z^{d})$, we have
			\begin{equation*}
				\sup_{\alpha >0}\alpha \left|\left\lbrace |\partial_{j}f(k)|\ge \frac{\alpha}{\langle k\rangle}\right\rbrace\right|\lesssim \|f\|_{\ell^{1,1}(\Z^{d})}.
			\end{equation*}
		\end{thm}
		\par 
		\begin{proof}
			By the symmetry, we only need to prove the case for $j=1$. We first prove the second statement. For $k=(k_{1},\cdots,k_{d}),$
			\begin{equation*}
				\left|\left\lbrace |\partial_{1}f(k)|\ge \alpha\langle k\rangle^{-1}\right\rbrace\right|=|\lbrace |\partial_{1}f(k)|\langle k\rangle^{2}\ge \alpha\langle k\rangle\rbrace|.
			\end{equation*}
			Since $(k_{1})^{2}\lesssim (k_{1}-m)^{2}+m^{2}$, we have 
			\begin{equation*}
				|\partial_{1}f(k_{1},k_{2},\cdots,k_{d})|\langle (k_{1},k_{2},\cdots,k_{d})\rangle^{2}\le\frac{4}{\pi}\sum_{m\in \Z}|f(m,k_{2},\cdots, k_{d})|\frac{k_{1}^{2}+\cdots+k_{d}^2+1}{|4(k_{1}-m)^2-1|}
			\end{equation*}
			\begin{equation*}
				\lesssim \sum_{m\in \Z}\frac{|f(m,k_{2},\cdots, k_{d})|}{|4(k_{1}-m)^2-1|}\left[((k_{1}-m)^2+1)+(k_{2}^{2}+\cdots+k_{d}^{2}+m^{2})\right].
			\end{equation*}
			Therefore, it suffices to prove the following (\ref{F1}) and (\ref{F2}) for any $\alpha >0$
			\begin{equation}\label{F1}
				\left|\left\lbrace \sum_{m\in \Z}\frac{|f(m,k_{2},\cdots, k_{d})|}{|4(k_{1}-m)^2-1|}\left[(k_{1}-m)^2+1\right]\ge \alpha \langle k\rangle \right\rbrace\right|\lesssim \|f\|_{\ell^{1,1}(\Z^{d})}/\alpha,
			\end{equation}
			\begin{equation}\label{F2}
				\left|\left\lbrace \sum_{m\in \Z}\frac{|f(m,k_{2},\cdots, k_{d})|}{|4(k_{1}-m)^2-1|}\left[k_{2}^{2}+\cdots+k_{d}^{2}+m^{2}\right]\ge \alpha \langle k\rangle \right\rbrace\right|\lesssim \|f\|_{\ell^{1,1}(\Z^{d})}/\alpha.
			\end{equation}
			For (\ref{F1}), by $\frac{(k_{1}-m)^2+1}{|4(k_{1}-m)^2-1|}\lesssim 1$, it suffices to prove the following statement
			\begin{equation}\label{F3}
				\left|\left\lbrace\sum_{m\in \Z}|f(m,k_{2},\cdots, k_{d})|\ge \alpha \langle k\rangle \right\rbrace\right|\lesssim \|f\|_{\ell^{1,1}(\Z^{d})}/\alpha.
			\end{equation}
			Set $\Gamma(k_{2},\cdots, k_{d}):=\sum_{m\in \Z}|f(m,k_{2},\cdots, k_{d})|.$ We can further decompose LHS of (\ref{F3}) as follows
			\begin{equation*}
				LHS=\sum_{(k_{2},\cdots, k_{d})\in \Z^{d-1}}|\lbrace \Gamma(k_{2},\cdots, k_{d})\ge \alpha \langle k\rangle \rbrace|_{k_{1}},
			\end{equation*}  
			where $A\subseteq \Z^{d}$, $|A|_{k_{1}}$ means the cardinality with respect to the first coordinate $k_{1}$,  for fixed $d-1$ coordinates $(k_{2},\cdots, k_{d}).$ We further have
			\begin{equation*}
				\lesssim \sum_{(k_{2},\cdots, k_{d})\in \Z^{d-1}}\Gamma(k_{2},\cdots, k_{d})/\alpha=\|f\|_{\ell^{1}(\Z^{d})}/\alpha\le \|f\|_{\ell^{1,1}(\Z^{d})}/\alpha,
			\end{equation*}
			as $\langle k\rangle\ge|k_{1}|$ and $\|f\|_{\ell^{1}(\Z^{d})}=\sum_{(k_{2},\cdots, k_{d})\in \Z^{d-1}}\Gamma(k_{2},\cdots, k_{d})$.
			\par 
			Next, we just need to prove (\ref{F2}), and we notice the following fact
			\begin{equation*}
				\sqrt{k_{2}^{2}+\cdots+k_{d}^2+m^{2}}\lesssim \sqrt{k_{1}^{2}+k_{2}^{2}+\cdots+k_{d}^2}+|k_{1}-m|\le \langle k \rangle+|k_{1}-m|.
			\end{equation*}
			Therefore, it suffices to prove following (\ref{F4}) and (\ref{F5})
			\begin{equation}\label{F4}
				\left|\left\lbrace \sum_{m\in \Z}\frac{|f(m,k_{2},\cdots, k_{d})|\sqrt{k_{2}^{2}+\cdots+k_{d}^2+m^{2}}}{|4(k_{1}-m)^2-1|}|k_{1}-m|\ge \alpha\langle k \rangle\right\rbrace\right|\lesssim \|f\|_{\ell^{1,1}(\Z^{d})}/\alpha,
			\end{equation}
			\begin{equation}\label{F5}
				\left|\left\lbrace \sum_{m\in \Z}\frac{|f(m,k_{2},\cdots, k_{d})|\sqrt{k_{2}^{2}+\cdots+k_{d}^2+m^{2}}}{|4(k_{1}-m)^2-1|}\ge \alpha\right\rbrace\right|\lesssim \|f\|_{\ell^{1,1}(\Z^{d})}/\alpha.
			\end{equation}
			For (\ref{F4}), we use the argument in the proof of (\ref{F1}) with $\frac{|k_{1}-m|}{|4(k_{1}-m)^2-1|}\lesssim 1 $. For (\ref{F5}), we write the sequence $(g_{k_{2},\cdots,k_{d}}(m))_{m\in \Z}:=(g(m,k_{2},\cdots,k_{d}))_{m\in\Z}$, for fixed $(k_{2},\cdots,k_{d}).$  In the following, we let $g(m,k_{2},\cdots,k_{d}):=|f(m,k_{2},\cdots, k_{d})|\sqrt{k_{2}^{2}+\cdots+k_{d}^2+m^{2}}$. Then we similarly decompose the LHS of (\ref{F5}) as in (\ref{F3})
			\begin{equation*}
				LHS=\sum_{(k_{2},\cdots, k_{d})\in \Z^{d-1}}\left|\left\lbrace \sum_{m\in \Z}\frac{g(m,k_{2},\cdots,k_{d})}{|4(k_{1}-m)^2-1|}\ge \alpha\right\rbrace\right|_{k1}
			\end{equation*}
			\begin{equation*}
				\lesssim \sum_{(k_{2},\cdots, k_{d})\in \Z^{d-1}}|\lbrace g_{k_{2},\cdots,k_{d}}\ast \varphi_1 \ge \alpha \rbrace|_{k1}\le \sum_{(k_{2},\cdots, k_{d})\in \Z^{d-1}}\|g_{k_{2},\cdots,k_{d}}\ast \varphi_1\|_{\ell^{1}(\Z)}/\alpha
			\end{equation*}
			\begin{equation*}
				\lesssim \sum_{(k_{2},\cdots, k_{d})\in \Z^{d-1}}\|g_{k_{2},\cdots,k_{d}}\|_{\ell^{1}(\Z)}/\alpha\le \|f\|_{\ell^{1,1}(\Z^{d})}/\alpha,
			\end{equation*}
			where the second inequality follows from Chebyshev's inequality, and the third inequality follows from Lemma \ref{Young}.
			\par 
			Now, we have proved the second statement of the theorem. The first statement will be a consequence of the second one. We first check the first statement for $p=\infty$. In the following, we first prove the case $d=1$ and then extend it to general $d$. In this case, we have
			\begin{equation*}
				\langle m \rangle|\partial_1 f(m)|=\left|\sum_{k\in \Z}\langle m \rangle f(k)\varphi(m-k)\right|\lesssim \sum_{k\in \Z}\frac{\langle m \rangle \|f\|_{\ell^{\infty, 1}(\Z)}}{\langle k \rangle  \langle m-k \rangle^{2}}
			\end{equation*}
			\begin{equation*}
				=\sum_{|k|\ge |m|/2}\frac{\langle m \rangle\|f\|_{\ell^{\infty, 1}(\Z)}}{\langle k \rangle  \langle m-k \rangle^{2}}+\sum_{|k|<|m|/2}\frac{\langle m \rangle\|f\|_{\ell^{\infty, 1}(\Z)}}{\langle k \rangle  \langle m-k \rangle^{2}}\equiv (I)+(II).
			\end{equation*}
			For $(I)$, we have  $\langle m \rangle \lesssim \langle k \rangle$, which implies
			\begin{equation*}
				(I)\lesssim \sum_{|k|\ge |m|/2}\frac{\|f\|_{\ell^{\infty, 1}(\Z)}}{\langle m-k \rangle^{2}}\lesssim \|f\|_{\ell^{\infty, 1}(\Z)}.
			\end{equation*}
			For $(II)$, we have $\langle m \rangle^{2}\lesssim \langle m-k \rangle^{2} $, which implies 
			\begin{equation*}
				(II)\lesssim \sum_{|k|<|m|/2}\frac{\|f\|_{\ell^{\infty, 1}(\Z)}}{\langle m \rangle \langle k \rangle}\lesssim \frac{\log(\langle m \rangle)}{\langle m \rangle}\|f\|_{\ell^{\infty, 1}(\Z)}\lesssim \|f\|_{\ell^{\infty, 1}(\Z)}. 
			\end{equation*}
			For general case, by the symmetry, we only need to prove the $\ell^{\infty,1}$-boundedness of $\partial_{1}$. Note that we have the following inequality
			\begin{equation*}
				\frac{\langle (m_{1},\cdots,m_{d}) \rangle}{\langle (k_{1},m_{2},\cdots,m_{d}) \rangle}\lesssim \max\left\lbrace \frac{\langle m_{1}\rangle}{\langle  k_{1}\rangle} ,1 \right\rbrace.
			\end{equation*}
			Applying this inequality, we can prove the $\ell^{\infty,1}$-boundedness of $\partial_{1}$ as in the proof of the case $d=1$.
			\par 
			Note that by definition the $\ell^{p,1}$-boundedness of $\partial_{j}$ is equivalent to the $\ell^{p}$-boundedness of the linear operator $T_{j}$, defined as
			\begin{equation*}
				T_{j}(g)(m):=\langle m \rangle\cdot \partial_{j}\left(\frac{g(m)}{\langle m \rangle}\right), \forall m\in\Z^{d}, j=1,\cdots,d.
			\end{equation*}
			Therefore, applying  Lemma \ref{interpolation} to the operator $T_{j}$, we prove the desired $\ell^{p,1}$-boundedness of the discrete partial derivative.  
		\end{proof}
		\begin{rem}
			The weak $\ell^{1,1}$-boundedness of the discrete partial derivative is sharp. In fact, we can choose $f=\delta_{0}$ and by simple calculation we show the unboundedness of $\|\partial_j f\|_{\ell^{1,1}(\Z^{d})}$ and the boundedness of $\| f\|_{\ell^{1,1}(\Z^{d})}$. This yields that the $\ell^{1,1}$-boundedness doesn't hold.
		\end{rem}
		\section{Energy estimates of discrete nonlinear wave equations}
		In this section, we study energy estimates of discrete wave equations. The conservation of the energy follows from the Noether theorem and Killing vector field theory \cite{11,17} without any surprise, since  discrete wave equations still have time translation invariance. 
		\par We first derive energy conservation of homogeneous linear discrete wave equation
		\begin{equation}\label{hwe1}
			\left\{
			\begin{aligned}
				& \partial_{t}^2 u(x,t) - \Delta u(x,t) = 0, \\
				& u(x,0) = f(x),\quad\partial_t u(x,0) = g(x),\quad (x,t)\in\Z^d\times \R.
			\end{aligned}
			\right.
		\end{equation}
		For the solution $u\in C^{2}([0,T];\ell^{2}(\Z^{d})),$ we define the energy as $$E_{disc}[t]:=\frac{1}{2}\sum_{k\in \Z^{d}}(|\partial_{t}u(k,t)|^{2}+|\partial u(k,t)|^{2}).$$	
		Since for any $v\in \ell^2(\Z^d)$, we have the following identity
		\begin{equation*}
			\sum_k |\partial v(k)|^{2}=\sum_k|D v(k)|^2,
		\end{equation*}
		the energy can be rewritten as
		$$E_{disc}[t]=\frac{1}{2}\sum_{k\in \Z^{d}}(|\partial_{t}u(k,t)|^{2}+|D u(k,t)|^{2}).$$

		\begin{thm}\label{ec1}
			
			For the solution $u\in C^{2}([0,T];\ell^{2}(\Z^{d}))$ of \eqref{hwe1}, we have the energy conservation
			\begin{equation*}
				E_{disc}[t]=E_{disc}[0],\quad \forall t\in[0,T].
			\end{equation*}
		\end{thm}
		\begin{proof}For simplicity, we only prove the case $d=1$ since the general case is similar. As the energy $E_{disc}[t]$ is of $\ell^{2}$-type, we apply the discrete Fourier transform, and it's equivalent to prove the following
			\begin{equation}\label{e1}
				\frac{1}{4\pi}\int_{0}^{2\pi}|\partial_{t}(\mathcal{F}(u)(x,t))|^{2}+\left|2i \sin\left(\frac{x}{2}\right)\right|^{2}\cdot |\mathcal{F}(u)(x,t)|^{2} dx \equiv C.
			\end{equation}
			Then, we differentiate (\ref{e1}) with respect to time $t$ and get
			\begin{equation*}
				\frac{1}{4\pi}\int_{0}^{2\pi}\partial_{t}^{2}(\mathcal{F}(u))\overline{\partial_{t}(\mathcal{F}(u))}+\partial_{t}(\mathcal{F}(u))\overline{\partial_{t}^{2}(\mathcal{F}(u))}+4\sin^{2}\left(\frac{x}{2}\right)\left[\partial_{t}(\mathcal{F}(u))\overline{\mathcal{F}(u)}+\mathcal{F}(u)\overline{(\partial_{t}\mathcal{F}(u))}\right]dx.
			\end{equation*}
			Note that by applying the discrete Fourier transform on the equation (\ref{hwe1}),  $$\partial_{t}^{2}\mathcal{F}(u)(x,t)+4\sin^{2}\left(\frac{x}{2}\right)\cdot \mathcal{F}(u)(x,t)=0.$$ Hence we deduce that (\ref{e1}) is constant, meaning that the energy $E_{disc}[t]$ is conserved.\end{proof}
		\begin{rem}
			The conservation of the energy $E_{disc}[t]$ indicates that the discrete partial derivative will play an important role in the analysis of wave equations.
		\end{rem}
		
		We study the energy conservation for the following discrete semilinear wave equation
		\begin{equation}\label{dfe1}
			\left\{
			\begin{aligned}
				& \partial_{t}^2 u(x,t) - \Delta u(x,t) = \mu |u|^{p-1}u,\ \mu=\pm 1, \\
				& u(x,0) = f(x),\quad\partial_t u(x,0) = g(x),\quad (x,t)\in\Z^d\times \R.
			\end{aligned}
			\right.
		\end{equation}
		We define the energy as 
		$$E_{disc}^{NLW}[t]:=\frac{1}{2}\left(\sum_{k\in \Z^{d}}(|\partial_{t}u(k,t)|^{2}+|\partial u(k,t)|^{2})-\frac{2\mu |u(k,t)|^{p+1}}{p+1}\right).$$
		\begin{thm}\label{ec2}
			For the solution $u\in C^{2}([0,T];\ell^{2}(\Z^{d}))$ of \eqref{dfe1}, we have the energy conservation
			\begin{equation*}
				E_{disc}^{NLW}[t]=E_{disc}^{NLW}[0],\quad \forall t\in[0,T].
			\end{equation*} 
			
		\end{thm}
		\begin{proof}
			The proof is similar with the proof of Theorem \ref{ec1}, hence we omit it.
		\end{proof}
		\par 
		Next, we derive some useful energy estimates, which will play important roles in our main result on the well-posedness of discrete nonlinear wave equations.
		\begin{thm}\label{energy estimate1}
			For inhomogeneous linear discrete wave equation 
			\begin{equation}\label{ihe1}
				\left\{
				\begin{aligned}
					& \partial_{t}^2 u(x,t) - \Delta u(x,t) = F(x,t), \\
					& u(x,0) = f(x),\quad\partial_t u(x,0) = g(x),\quad (x,t)\in\Z^d\times \R,
				\end{aligned}
				\right.
			\end{equation}
			if $u\in C^{2}([0,T];\ell^{2}(\Z^{d}))$ is the solution of (\ref{ihe1}) and $f,g\in \ell^{2,k}(\Z^{d})$, $F\in L^{1}([0,T];\ell^{2,k}(\Z^{d})),$ for $k=0,1$, then we have the explicitly time-dependent energy estimate, $\forall t\in[0,T],$
			\begin{equation}\label{ee1}
				\|u(\cdot,t)\|_{\ell^{2,k}(\Z^{d})}+\|\partial_{t}u(\cdot,t)\|_{\ell^{2,k}(\Z^{d})}\lesssim (1+|t|^{k+1})\left[\|f\|_{\ell^{2,k}(\Z^{d})}+\|g\|_{\ell^{2,k}(\Z^{d})}+\int_{0}^{t}\|F(\cdot,s)\|_{\ell^{2,k}(\Z^{d})}ds\right].
			\end{equation}
			
		\end{thm}
		\begin{proof}
			By Duhamel's formula, we have 
			\begin{equation*}
				\mathcal{F}(u)(x,t)=\mathcal{F}(f)(x)\cdot \cos(K(x)t)+\mathcal{F}(g)(x)\frac{\sin(K(x)t)}{K(x)}+\int_{0}^{t}\frac{\sin(K(x)(t-s))}{K(x)}\mathcal{F}(F)(x,s)ds,
			\end{equation*}
			where $K(x)=\sqrt{\sum_{j=1}^{d}2-2\cos(x_{j})}=2\sqrt{\sum_{j=1}^{d}\sin^{2}(\frac{x_{j}}{2})}$.
			\par 
			For the case $k=0$, we apply the discrete Fourier transform on (\ref{ee1}) and get
			\begin{equation*}
				LHS=\|\mathcal{F}(u)(\cdot,t)\|_{L^{2}(\mathbb{T}^{d})}+\|\partial_{t}(\mathcal{F}(u))(\cdot,t)\|_{L^{2}(\mathbb{T}^{d})}.
			\end{equation*}
			We only deal with the estimate for the term $\|\mathcal{F}(u)(\cdot,t)\|_{L^{2}(\mathbb{T}^{d})},$ since the estimate for  $\|\partial_{t}(\mathcal{F}(u))(\cdot,t)\|_{L^{2}(\mathbb{T}^{d})}$ is similar. For $\|\mathcal{F}(u)(\cdot,t)\|_{L^{2}(\mathbb{T}^{d})}$, we have
			\begin{equation*}
				\|\mathcal{F}(u)(\cdot,t)\|_{L^{2}(\mathbb{T}^{d})}\le \|\mathcal{F}(f)(\cdot)\cdot \cos(K(\cdot)t)\|_{L^{2}(\mathbb{T}^{d})}+\left\|\mathcal{F}(g)(\cdot)\frac{\sin(K(\cdot)t)}{K(x)}\right\|_{L^{2}(\mathbb{T}^{d})}
			\end{equation*}
			\begin{equation*}
				+\left\|\int_{0}^{t}\frac{\sin(K(\cdot)(t-s))}{K(\cdot)}\mathcal{F}(F)(\cdot,s)ds\right\|_{L^{2}(\mathbb{T}^{d})}.
			\end{equation*}
			By the observation that $\left|\frac{\sin(K(x)t)}{K(x)}\right|\le t$ and the Minkowski inequality, we derive the following
			\begin{equation*}
				\le \|\mathcal{F}(f)(\cdot)\|_{L^{2}(\mathbb{T}^{d})}+t\|\mathcal{F}(g)(\cdot)\|_{L^{2}(\mathbb{T}^{d})}+t\int_{0}^{t}\|\mathcal{F}(F)(\cdot,s)\|_{L^{2}(\mathbb{T}^{d})}ds
			\end{equation*}
			\begin{equation*}
				\lesssim (1+|t|)\left[\|f\|_{\ell^{2}(\Z^{d})}+\|g\|_{\ell^{2}(\Z^{d})}+\int_{0}^{t}\|F(\cdot,s)\|_{\ell^{2}(\Z^{d})}ds\right].
			\end{equation*}
			\par
			For the case $k=1$, we still apply the discrete Fourier transform on (\ref{ee1}) and get
			\begin{equation*}
				LHS\approx\|\mathcal{F}(u)(\cdot,t)\|_{H^{1}(\mathbb{T}^{d})}+\|\partial_{t}\mathcal{F}(u)(\cdot,t)\|_{H^{1}(\mathbb{T}^{d})}.
			\end{equation*}
			By direct calculation, we have
			\begin{equation*}
				|\partial_{x_j}K(x)|=\left|\frac{\sin(\frac{x_{j}}{2})\cos(\frac{x_{j}}{2})}{\sqrt{\sum_{j=1}^{d}\sin^{2}(\frac{x_{j}}{2})}}\right|\le 1,
			\end{equation*}
			\begin{equation*}
				\left|\partial_{x_j}\frac{\sin(K(x)t)}{K(x)}\right|=|H'(K(x))\cdot \partial_{x_j}K(x)|\le |H'(K(x))|,
			\end{equation*}
			where $H(y):=\frac{\sin(ty)}{y}$ is a smooth function. Then the estimate of the second case is similar with the first case, and hence we omit it.
		\end{proof} 
		\begin{rem}
			The reason why we call Theorem \ref{energy estimate1} the explicitly time-dependent energy estimate is that the coefficient $1+|t|^{k+1}$ is related to time variable $t$. Besides, it's worthy to remember that this energy estimate is only suitable for the traditional case, which means the equation involves $\Delta$ purely.
		\end{rem}
		Next, we derive an implicitly time-dependent energy estimate that applies in the generalized case where the coefficients of second-order discrete partial derivatives are functions. However, this generalized energy estimate introduces an additional exponential term.
		\begin{thm}\label{energy estimate2}
			For a generalized d'Alembert operator $\Box_{g}:=\partial_{t}^{2}-g^{jk}\partial_{jk}$, we have the following energy estimate of $u\in C^{2}([0,T];\ell^{2,k}(\Z^{d}))$, with $g^{jk}\in L^1([0,T];\ell^{\infty}(\Z^{d}))$, $C>0$, $k=0,1$, $\forall t\in [0,T]$
			\begin{equation*}
				\|u(\cdot,t)\|_{\ell^{2,k}(\Z^{d})}+\|\partial_{t}u(\cdot,t)\|_{\ell^{2,k}(\Z^{d})}\lesssim\left(\|u(\cdot,0)\|_{\ell^{2,k}(\Z^{d})}+\|\partial_{t}u(\cdot,0)\|_{\ell^{2,k}(\Z^{d})}+\int_{0}^{t}\|\Box_{g}u(\cdot,s)\|_{\ell^{2,k}(\Z^{d})}ds\right)
			\end{equation*}
			\begin{equation*}
				\times exp\left(C\int_{0}^{t}\left(\sum_{j,k=1}^{d}\|g^{jk}(\cdot,s)\|_{\ell^{\infty}(\Z^{d})}+1\right)ds\right).
			\end{equation*}
		\end{thm}
		\begin{proof}
			For the case $k=0$, we consider the energy $E(t)$ defined as			\begin{equation*}
				E(t):=\sum_{m\in \Z^{d}}|u(m,t)|^{2}+|\partial_{t}u(m,t)|^{2}.
			\end{equation*} 
			For simplicity, we deal with the case that involved functions are real-valued. We differentiate $E(t)$ with respect to time $t$ and get
			\begin{equation*}
				\frac{d}{dt}E(t)=2\sum_{m\in\Z^{d}} u(m,t)\partial_{t}u(m,t)+\partial_{t}u(m,t)[\Box_{g}u(m,t)+g^{jk}(m,t)\partial_{jk}u(m,t)]
			\end{equation*}
			\begin{equation*}
				\lesssim E(t)+E(t)^{\frac{1}{2}}\cdot \|\Box_{g}u(\cdot,t)\|_{\ell^{2}(\Z^{d})}+\sum_{j,k=1}^{d}\|g^{jk}(\cdot,t)\|_{\ell^{\infty}(\Z^{d})}\cdot E(t).
			\end{equation*}
			Then we divide by $E(t)^{\frac{1}{2}}$ on both sides, and apply the Gronwall inequality to prove this estimate. For the case $k=1$, since the proof is similar, we omit it.
		\end{proof} 
		\begin{rem}\label{L}
			The name of implicitly time-dependent energy estimate comes from the coefficient is time-independent, which informally implies the usefulness in long-time well-posedness theory.  Moreover, this estimate and its proof are evidently applicable to a broader class of operators $L:=\partial_{t}^{2}-g^{jk}\partial_{jk}+b\partial_{t}+b^{j}\partial_{j}+cu$, with some simple modification.
			
		\end{rem}

		\section{local well-posedness of discrete nonlinear wave equations}
		We shall first establish the global well-posedness theory for discrete generalized linear wave equation given by
		\begin{equation}\label{generalized}
			\left\{
			\begin{aligned}
				& \partial_{t}^2 u(x,t) - g^{jk}(x,t)\partial_{jk} u(x,t) = F(x,t), \\
				& u(x,0) = f(x),\quad\partial_t u(x,0) = g(x),\quad (x,t)\in\Z^d\times \R.
			\end{aligned}
			\right.
		\end{equation}
		The reason why we call it's ``generalized", is that this equation generalizes the equation (\ref{ihe1}) and has no explicit formula to ensure its existence. Before we prove the global well-posedness of this equation, we need to introduce the definition of weak solutions, which is essential in our proofs.
		\begin{defi}
			We say that $u$ is a weak solution for the equation (\ref{generalized}) with initial data $(f,g)=(0,0)$, if it satisfies the following equation for any $\phi\in C_{0}^{\infty}(\Z^{d}\times (0,T))$
			\begin{equation*}
				\int_{0}^{T}\sum_{m\in \Z^{d}}F(m,t)\overline{\phi(m,t)}dt=\int_{0}^{T}\sum_{m\in \Z^{d}} u(m,t)\overline{L^{\ast}\phi(m,t)}dt,
			\end{equation*}
			where $L^{\ast}\phi(x,t):=\partial_{t}^{2}\phi(x,t)-\partial_{jk}(\overline{g^{jk}(x,t)}\phi(x,t))$.
		\end{defi}
		Now, we are ready to state our result.
		\begin{thm}\label{ge}
			If $g^{jk}\in L^{\infty}([0,T];\ell^{\infty}(\Z^{d}))\cap C^{0}(\left[0,T\right];\ell^{\infty}(\Z^{d})) $, then for $f, g\in \ell^{2,k}(\Z^{d})$ and $F\in L^{1}(\left[0,T\right];\ell^{2,k}(\Z^{d}))\cap C^{0}(\left[0,T\right];\ell^{2,k}(\Z^{d})) $, $T>0$, there exists a unique classical solution $u\in C^{2}(\left[0,T\right];\ell^{2,k}(\Z^{d}))$ for the equation (\ref{generalized}).
		\end{thm}
		\begin{proof}
			The uniqueness part follows from the energy estimate. For the existence part, it suffices to consider the case when $f=g=0$, otherwise we consider $\widetilde{u}=u-(f+tg)$ instead. First, we consider the case $k=0$. We claim the following estimate
			\begin{equation}\label{el}
				\|\phi(\cdot,t)\|_{\ell^{2}(\Z^{d})}\lesssim_{T}\int_{0}^{T}\|L^{\ast}\phi(\cdot,s)\|_{\ell^{2}(\Z^{d})}ds,\; \forall \phi \in C_{0}^{\infty}(\Z^{d}\times (0,T)), \forall t\in [0,T].
			\end{equation}
			It suffices to establish the following energy estimate for the operator $L^{\ast}$ and any $u\in C^{2}(\left[0,T\right];\ell^{2}(\Z^d))$
			\begin{equation*}
				\|u(\cdot,t)\|_{\ell^{2}(\Z^{d})}+\|\partial_{t}u(\cdot,t)\|_{\ell^{2}(\Z^{d})}\lesssim_{T}\left(\|u(\cdot,0)\|_{\ell^{2}(\Z^{d})}+\|\partial_{t}u(\cdot,0)\|_{\ell^{2}(\Z^{d})}+\int_{0}^{t}\|L^{\ast}u(\cdot,s)\|_{\ell^{2}(\Z^{d})} ds\right ).
			\end{equation*}
			Considering $$E(t):=\sum_{m\in\Z^{d}}|u(m,t)|^{2}+|\partial_{t}u(m,t)|^{2},$$ we differentiate this energy with respect to time $t$ and we deal with the real-valued case for simplicity. Hence,
			\begin{equation*}
				\frac{d}{dt}E(t)=\sum_{m\in\Z^{d}}u(m,t)\partial_{t}u(m,t)+\partial_{t}u(m,t)L^{\ast}u(m,t)+\partial_{t}u(m,t)\partial_{jk}(\overline{g^{jk}}u)(m,t)
			\end{equation*} 
			\begin{equation*}
				\lesssim E(t)^{\frac{1}{2}}\|L^{\ast}u(\cdot,t)\|_{\ell^{2}(\Z^{d})}+\left(\sum_{j,k=1}^{d}\|g^{jk}(\cdot,t)\|_{\ell^{\infty}(\Z^{d})}+1\right)E(t).
			\end{equation*}
			Dividing by $E(t)^{\frac{1}{2}}$ on both sides, we immediately get the estimate from the Gronwall inequality and the uniformly boundedness of $g^{jk}$.
			\par For the existence of weak solutions, we define the following linear space $V$ and a linear functional $\ell_{F}$ on it
			\begin{equation*}
				V:=L^{\ast}C_{0}^{\infty}(\Z^{d}\times (0,T))=\lbrace L^{\ast}v; v\in C_{0}^{\infty}(\Z^{d}\times (0,T)) \rbrace,
			\end{equation*} 
			\begin{equation*}
				\ell_{F}(L^{\ast}v):=\int_{0}^{T}\sum_{m\in\Z^{d}}F(m,t)\overline{v(m,t)}dt.
			\end{equation*}
			According to the estimate (\ref{el}) and Cauchy-Schwarz inequality, we derive the following
			\begin{equation*}
				|\ell_{F}(L^{\ast}v)|\le \int_{0}^{T}\|F(\cdot,t)\|_{\ell^{2}(\Z^{d})}\|v(\cdot,t)\|_{\ell^{2}(\Z^{d})}dt\lesssim_{T,F}\int_{0}^{T}\|L^{\ast}v(\cdot,t)\|_{\ell^{2}(\Z^{d})}dt.
			\end{equation*}
			Then we can regard $V$ as a subspace of $L^{1}(\left[0,T\right];\ell^{2}(\Z^{d}))$. Then by Hahn-Banach Theorem, $\ell_{F}$ can be extended to $L^{1}(\left[0,T\right];\ell^{2}(\Z^{d}))$. Since the dual space of $L^{1}(\left[0,T\right];\ell^{2}(\Z^{d}))$ is $L^{\infty}(\left[0,T\right];\ell^{2}(\Z^{d}))$, there exists $u\in L^{\infty}(\left[0,T\right];\ell^{2}(\Z^{d})),$ which is the weak solution for the equation (\ref{generalized}). We also notice that 
			\begin{equation}\label{x}
				\partial_{t}^2 u(x,t)= g^{jk}(x,t)\partial_{jk} u(x,t)+F(x,t),
			\end{equation}
			where the LHS is understood as taking weak derivatives and the RHS is in $L^{\infty}(\left[0,T\right];\ell^{2}(\Z^{d}))$, which imply that $\partial_{t}u\in C([0,T];\ell^{2}(\Z^{d}))$ and $u\in C^{1}([0,T];\ell^{2}(\Z^{d}))$. Applying (\ref{x}) again, we conclude that $u\in C^{2}([0,T];\ell^{2}(\Z^{d})$ is actually a classical solution to the equation (\ref{generalized}). For the case $k=1$, we can similarly get the estimate 			\begin{equation*}
				\|\phi(\cdot,t)\|_{\ell^{2,-1}(\Z^{d})}\lesssim_{T}\int_{0}^{T}\|L^{\ast}\phi(\cdot,t)\|_{\ell^{2,-1}(\Z^{d})}dt,\; \forall \phi \in C_{0}^{\infty}(\Z^{d}\times (0,T)), \forall t\in [0,T],		
			\end{equation*}
			\begin{equation*}
				|\ell_{F}(L^{\ast}v)|\lesssim_{T,F}\int_{0}^{T}\|L^{\ast}v(\cdot,t)\|_{\ell^{2,-1}(\Z^{d})}dt.
			\end{equation*}
			Then by same arguments, we can prove the existence of a solution $u\in C^{2}([0,T];\ell^{2,1}(\Z^{d}))$.	
		\end{proof}

		\begin{rem}
			Based on Remark \ref{L}, we can similarly establish the global well-posedness theory for the following general equation, compared with the equation (\ref{generalized})
			\begin{equation*}
				\left\{
				\begin{aligned}
					& Lu(x,t)=F(x,t), \\
					& u(x,0) = f(x),\quad\partial_t u(x,0) = g(x),\quad (x,t)\in\Z^d\times \R.
				\end{aligned}
				\right.
			\end{equation*}
			The operator $L$ is defined in Remark \ref{L}.
		\end{rem}
		Now, we are ready to prove one of our main results, that is, the local well-posedness of quasilinear discrete wave equation (\ref{equation1}).
		\par 
		\begin{proof}[Proof of Theorem \ref{Th1}] For clarity, we divide the proof into 5 steps.

			\noindent
			\textbf{Step 1:}
			We first consider iteration argument as follows. Set $u_{-1}\equiv 0,$ and consider
			\begin{equation}\label{iteration}
				\left\{
				\begin{aligned}
					& \partial_{t}^2 u_{m}(x,t) - g^{jk}(u_{m-1},u'_{m-1})\partial_{jk} u_{m}(x,t) = F(u_{m-1},u'_{m-1}), \\
					& u_{m}(x,0) = f(x),\quad\partial_t u_{m}(x,0) = g(x),\quad (x,t)\in\Z^d\times \R.
				\end{aligned}
				\right.
			\end{equation}
			Then, with the results of discrete generalized linear wave equations in Theorem \ref{ge}, we see that $u_{m}$ defined above is a classical solution, that is, $u_{m}\in C^{2}(\left[0,T\right];\ell^{2,k}(\Z^{d})).$
			\par 
			Next, we introduce the key energy $A_{m}(t)$ as follows
			\begin{equation*}
				A_{m}(t):=\|u(\cdot,t)\|_{\ell^{2,k}(\Z^{d})}+\|\partial_{t}u(\cdot,t)\|_{\ell^{2,k}(\Z^{d})}.
			\end{equation*}
			We claim that there exist $T>0$ and $M>0$, such that $\forall m, \forall t\in \left[0,T\right]$,  $A_{m}(t)\le M $.
			
			\noindent
			\textbf{Step 2:}
			We prove the above claim by induction. Fix $M$ big enough such that $$\|f\|_{\ell^{2,k}(\Z^{d})}, \|g\|_{\ell^{2,k}(\Z^{d})}\ll M.$$ We can ensure $A_{0}(t) \le M, \forall t\in [0,1]$, from the energy estimate in Theorem \ref{energy estimate2}. Suppose that there exists $0<T<1$, such that for $m\le n-1$, $\forall t\in \left[0,T\right]$,  $A_{m}(t)\le M $, then we will show that the claim still holds for $m=n$.
			\par 
			Applying the energy estimate in Theorem \ref{energy estimate2}, we derive the following 
			\begin{equation}\label{o}
				A_{n}(t)\lesssim \left(A_{n}(0)+\int_{0}^{t}\|F(u_{n-1},u'_{n-1})\|_{\ell^{2,k}(\Z^{d})}ds\right)
				\cdot exp\left(C\int_{0}^{t}\;\left(\sum_{j,k=1}^{d}\|g^{jk}(u_{n-1},u'_{n-1})\|_{\ell^{\infty}(\Z^{d})}+1 \right)ds\right).
			\end{equation}
			\par 
			Based on the hypothesis on $F$, $g^{jk}$ and the induction assumption, we have the estimates
			\begin{equation*}
				|F(u_{n-1},u'_{n-1})|\lesssim_{M}|u_{n-1}|+|u'_{n-1}|; \quad \sum_{j,k=1}^{d}\|g^{jk}(u_{n-1},u'_{n-1})\|_{\ell^{\infty}(\Z^{d})}\lesssim_{M}1.
			\end{equation*}
			Based on the above estimates, we finally derive 
			\begin{equation*}
				A_{n}(t)\le C_{1}(A_{n}(0)+C_{2}Mt)\cdot e^{C_{3}t},
			\end{equation*}
			where $C_{1}$ is independent on $M$ and $C_{2}, C_{3}$ depend on $M$, which are bounded when $ M$ is bounded. Another important observation is that $A_{n}(0)$ is only dependent on initial data $f,g$. Therefore, we can let $T\ll M,1$, and then we can conclude that $A_{n}\le M$ and the induction is complete.
			
			\noindent
			\textbf{Step 3:} Next, we prove that $\lbrace u_{m}\rbrace$ is a Cauchy sequence in $C^{1}(\left[0,T\right];\ell^{2,k}(\Z^{d}))$. According to the iteration in (\ref{iteration}), we have the following 
			\begin{equation*}
				\partial_{t}^{2}(u_{m}-u_{m-1})(x,t)-g^{jk}(u_{m-1},u'_{m-1})\partial_{jk}(u_{m}-u_{m-1})(x,t)=(\ast),
			\end{equation*}
			\begin{equation*}
				(\ast)= [g^{jk}(u_{m-2},u'_{m-2})-g^{jk}(u_{m-1},u'_{m-1})]\partial_{jk}u_{m-1}(x,t)+F(u_{m-1},u'_{m-1})-F(u_{m-2},u'_{m-2}).
			\end{equation*}
			Based on the assumptions of $F,g^{jk}$, we have a similar estimate
			\begin{equation*}
				(\ast)=O_{M}(|u_{m-1}-u_{m-2}|+|u'_{m-1}-u'_{m-2}|).
			\end{equation*}
			To derive the Cauchy property of $\lbrace u_{m}\rbrace$, we introduce $C_{m}(t)$ defined as 
			\begin{equation*}
				C_{m}(t):=\|u_{m}(\cdot,t)-u_{m-1}(\cdot,t)\|_{\ell^{2,k}(\Z^d)}+\|\partial_{t}u_{m}(\cdot,t)-\partial_{t}u_{m-1}(\cdot,t)\|_{\ell^{2,k}(\Z^d)}.
			\end{equation*}
			Based on the above estimate and the fact that the initial data of $u_{m}-u_{m-1}$ is $0$, we apply the energy estimate in Theorem \ref{energy estimate2} and get
			\begin{equation}\label{t}
				C_{m}(t)\le C\int_{0}^{t}C_{m-1}(\tau)d\tau.
			\end{equation}
			Then we apply (\ref{t}) for $m$ times and derive the following
			\begin{equation*}
				C_{m}(t)\le C^{m}\int\int \cdots \int_{0\le \tau_{1} \le \cdots \le\tau_{m}\le t}C_{0}(\tau_{1})d\tau_{1} \cdots \tau_{m}\lesssim_{M}\frac{(Ct)^{m}}{m!}.
			\end{equation*}
			Therefore $\lbrace u_{m}\rbrace$ is a Cauchy sequence. We write the limit of the sequence as $u$, then $u\in C^{1}(\left[0,T\right];\ell^{2,k}(\Z^{d}))$. Based on the iteration (\ref{iteration}) and the fact that $\lbrace u_{m}\rbrace$ is a Cauchy sequence in $ C^{1}(\left[0,T\right];\ell^{2,k}(\Z^{d}))$, we immediately conclude that
			\begin{equation*}
				\lbrace\partial_{t}^{2}u_{m}(x,t)\rbrace=\lbrace{g^{jk}(u_{m-1},u'_{m-1})\partial_{jk} u_{m}(x,t)+ F(u_{m-1},u'_{m-1})}\rbrace
			\end{equation*}
			is also a Cauchy sequence in $C(\left[0,T\right];\ell^{2,k}(\Z^{d}))$. Combined it with the above results, we deduce that $u$ is a classical solution of the equation (\ref{equation1}).

			\noindent
			\textbf{Step 4:} The uniqueness follows from the above analysis. In fact, suppose that there is another solution $\widetilde{u}$, then we define $C(t)$ as
			\begin{equation*}
				C(t):=\|u(\cdot,t)-\widetilde{u}(\cdot,t)\|_{\ell^{2,k}(\Z^{d})}+\|\partial_{t}u(\cdot,t)-\partial_{t}\widetilde{u}(\cdot,t)\|_{\ell^{2,k}(\Z^{d})}.
			\end{equation*}
			Similarly, applying the energy estimate, we can get $C(t)\le C\int_{0}^{t}C(\tau)d\tau $. Then the Gronwall inequality shows that $C(t)\equiv 0$, which implies $u\equiv\widetilde{u}$.

			\noindent
			\textbf{Step 5:} Finally, we derive the continuation criterion. From the above arguments, one easily sees that if $\|u(\cdot,t)\|_{\ell^{2,k}(\Z^{d})}+\|\partial_{t}u(\cdot,t)\|_{\ell^{2,k}(\Z^{d})}$ is bounded in $\left[0,T^{\ast}\right)$, then we can extend the solution $u$ over $T^{\ast}$. We claim that the weaker requirement $\|u(\cdot,t)\|_{\ell^{\infty}(\Z^{d})}+\|\partial_{t}u(\cdot,t)\|_{\ell^{\infty}(\Z^{d})}$ is bounded in $\left[0,T^{\ast}\right)$ can imply the above stronger requirement. 
			\par 
			Letting $n\to \infty$ in (\ref{o}), we have the following energy estimate for $u$
			\begin{equation}\label{w}
				A(t)\lesssim \left(A(0)+\int_{0}^{t}\|F(u,u')\|_{\ell^{2,k}(\Z^{d})}ds\right)\cdot exp\left(C\int_{0}^{t}\left(\sum_{j,k=1}^{d}\|g^{jk}(u,u')\|_{\ell^{\infty}(\Z^{d})}+1\right)ds\right),
			\end{equation}
			where $A(t):=\|u(\cdot,t)\|_{\ell^{2,k}(\Z^{d})}+\|\partial_{t}u(\cdot,t)\|_{\ell^{2,k}(\Z^{d})}$.
			The key observation is that the estimates below
			\begin{equation}\label{q}
				|F(u,u')|\lesssim_{M}|u|+|u'|; \quad	\sum_{j,k=1}^{d}\|g^{jk}(u,u')\|_{\ell^{\infty}(\Z^{d})}\lesssim_{M}1
			\end{equation}
			only require that $\|u(\cdot,t)\|_{\ell^{\infty}(\Z^{d})}+\|\partial_{t}u(\cdot,t)\|_{\ell^{\infty}(\Z^{d})}$ is bounded in $\left[0,T^{\ast}\right)$. Substituting the estimate (\ref{q}) into (\ref{w}), we have
			\begin{equation*}
				A(t)\lesssim \left(A(0)+\int_{0}^{t} A(s)ds\right)\cdot e^{Ct}\lesssim_{T^{\ast}}\left(A(0)+\int_{0}^{t} A(s)ds\right).
			\end{equation*}
			Then applying the Gronwall inequality, we deduce that $A(t)$ is bounded in finite time, which implies $\|u(\cdot,t)\|_{\ell^{2,k}(\Z^{d})}+\|\partial_{t}u(\cdot,t)\|_{\ell^{2,k}(\Z^{d})}$ is bounded in $\left[0,T^{\ast}\right),$ and the continuation criterion is proved.
		\end{proof}
		\begin{rem}
			In fact, by the fundamental theorem of calculus, the continuation criterion can be further weaken to $\|\partial_{t}u(\cdot,t)\|_{\ell^{\infty}(\Z^{d})}$ is bounded in $\left[0,T^{\ast}\right)$. However, in practice, $\|\partial_{t}u(\cdot,t)\|_{\ell^{\infty}(\Z^{d})}$ is less useful than $\|\partial_{t}u(\cdot,t)\|_{\ell^{2}(\Z^{d})},$ and $\|\partial_{t}u(\cdot,t)\|_{\ell^{2}(\Z^{d})}$ is less useful than $\|u(\cdot,t)\|_{\ell^{2}(\Z^{d})}+\|\partial_{t}u(\cdot,t)\|_{\ell^{2}(\Z^{d})}$, as the latter is applicable for the energy estimate.
		\end{rem}
		\section{Long time \& global well-posedness of discrete nonlinear wave equations}
		We first give the proof of Theorem \ref{Th2}, which ensures the long time well-posedness for quasilinear discrete wave equations with small initial data.
		\par 
		\begin{proof}[Proof of Theorem \ref{Th2}]
			The key is to use the continuation criterion for extending the classical solution. For simplicity, we only prove the case when $g^{jk}(u,u')=g^{jk}(u')$ and $F(u,u')=F(u')$, since the general case is similar. 
			\par As before, we consider the energy $A(t):=\|u(\cdot,t)\|_{\ell^{2}(\Z^{d})}+\|\partial_{t}u(\cdot,t)\|_{\ell^{2}(\Z^{d})},$ and we only need to show that this energy stays bounded at least before $K\cdot \log\left(\log\left(\frac{1}{\varepsilon}\right)\right)$. In the following proof, we may sometime omit some irrelevant constants. We first introduce the constant 
			\begin{equation*}
				M:=\max\left\lbrace \max_{j,k;|x|\le 1}|g^{jk}(x)| , \\ \max_{|x|\le 1}|\nabla F(x)| , 100  , d\right\rbrace.
			\end{equation*}
			
			\par We first notice that, if for time interval $[0,T]$, there exist $R>0$ s.t. $A(0)<\frac{R}{2}$ and property(P) that $\forall t\in [0,T]$, $A(t)\le R$ would imply $A(t)\le \frac{R}{2}$, then we can obtain $A(t)\le R, \forall t\in [0,T]$.
			Therefore, the energy $A(t)$ is bounded in $[0,T]$, which, by continuation criterion, will ensure the existence  of the solution in $[0,T]$. Thus, in the following, we choose $A$ satisfying $1\ll_{f,g} A$(to ensure $A(0)<\frac{A\varepsilon}{2}$) and  $R=A\varepsilon=\frac{1}{10}$ (this can be done if $\varepsilon$ is small enough). The whole theorem comes to show that for $T\lesssim \log(\log(\frac{1}{\varepsilon}))$, the property(P) is true on $[0,T]$. 
			\par From the energy estimate from Theorem \ref{energy estimate2} and the choice of $A$, we derive 
			\begin{equation}\label{k}
				A(t)\le C\left(A(0)+\int_{0}^{t}\|F(u'(\cdot,s))\|_{\ell^{2}(\Z^{d})}ds\right) \cdot exp\left(C'\int_{0}^{t}\left(\sum_{j,k=1}^{d}\|g^{jk}(u'(\cdot,s))\|_{\ell^{\infty}(\Z^{d})}+1\right)ds\right)
			\end{equation}
			and 
			\begin{equation}\label{h}
				\|g^{jk}(u'(\cdot,t))\|_{\ell^{\infty}(\Z^{d})}\lesssim M;\; \|F(u'(\cdot,t))\|_{\ell^{2}(\Z^{d})}\le M\|u'(\cdot,t)\|_{\ell^{2,k}(\Z^{d})}\lesssim M A(t).
			\end{equation}
			Then substituting above estimates into (\ref{k}), we immediately have
			\begin{equation*}
				A(t)\le Ce^{C'MT}\left(A(0)+\int_{0}^{t} M A(s)ds\right)\le CMe^{C'MT}\left(A(0)+\int_{0}^{t}  A(s)ds\right).
			\end{equation*}
			For simplicity, we write  $F:=CMe^{C'MT}$. Applying the Gronwall inequality, we derive
			\begin{equation*}
				A(t)\le Fe^{Ft}A(0)\le Fe^{FT}A(0).
			\end{equation*}
			Therefore, if we have $Fe^{FT}A(0)\le \frac{A\varepsilon}{2}$, then we prove the property (P), which shows that the energy $A(t)$ is bounded in $\left[0,T\right)$. 
			\par 
			Note that $$Fe^{FT}A(0)\le \frac{A\varepsilon}{2}\iff Fe^{FT}\lesssim \frac{A}{2} \iff T\lesssim \log(\log(A))\approx \log(\log(\frac{1}{\varepsilon})),$$ where the second equivalence follows from applying $\log(\log(\cdot))$ on both sides and omitting small quantity. Therefore, we derive the existence of the solution in $[0,K\cdot \log(\log(\frac{1}{\varepsilon}))]$, which shows $\log(\log(\frac{1}{\varepsilon}))\lesssim T^{\ast}$.
		\end{proof}
		
		\begin{rem}
			The  lower bound  for maximal existence time in the above theorem is of $\log(\log(\cdot)) $-type, which is a consequence of the exponential term in the energy estimate and, most importantly, $A\varepsilon\lesssim 1$ to ensure (\ref{h}). Therefore, if we require more conditions on $g^{jk}$ and $F$, we can derive stronger lower bound for the maximal existence time.
		\end{rem}
		In the following, we assume that $g^{jk}=\delta_{jk}$, which is the traditional version of discrete nonlinear wave equations. We use the energy estimate in Theorem \ref{energy estimate1} with $k=0$, which has no exponential term.
		\begin{thm}
			If $g^{jk}=\delta_{jk}$, then we have $T^{\ast}\ge K\cdot \sqrt{\log(\dfrac{1}{\varepsilon})}$, where $K=K(f,g,d,F)$ is a positive constant.
		\end{thm}
		\begin{proof}
			Following the continuity method of Theorem \ref{Th2}, we apply the energy estimate in Theorem \ref{energy estimate1} with $k=0$ and choose $A$ such that $A\varepsilon=\frac{1}{10}$ to guarantee (\ref{h}), then we have
			\begin{equation*}
				A(t)\le C(t)\left(A(0)+\int_{0}^{t}A(s)ds\right),
			\end{equation*}
			where $C(t)=C\cdot(1+t)$. Applying the Gronwall inequality, we can similarly derive the lower bound $T^{\ast}\ge K\cdot \sqrt{\log(\dfrac{1}{\varepsilon})}$.
		\end{proof}
		Next, for the equation (\ref{dfe1}), we will derive the long-time well-posedness for focusing case ($\mu=1$) with small data and the global-wellposedness for defocusing case ($\mu=-1$) with large data.
		\begin{thm}\label{j}
			If $g^{jk}=\delta_{jk}$, $F(u)=\mu |u|^{p-1}u, \mu=1, p>1$, then we have $T^{\ast}\ge K\cdot (\frac{1}{\varepsilon})^{\frac{p-1}{p+1}}$, where $K=K(p,f,g,d)$ is a constant.
		\end{thm}
		\begin{proof}
			We directly apply the energy estimate in Theorem \ref{energy estimate1} with $k=0$ and get
			\begin{equation}\label{s}
				A(t)\le C(t)\left(A(0)+\int_{0}^{t}A(s)^{p}ds\right).
			\end{equation} 
			Letting $\theta (t):=A(0)+\int_{0}^{t}A(s)^{p}ds$, we can rewrite (\ref{j}) as $\theta'(t)\le C(t)^{p}\theta(t)^{p}$, 
			then we immediately get the following estimate
			\begin{equation*}
				 \frac{1}{A(0)^{p-1}}-\frac{1}{\theta(t)^{p-1}}\lesssim \int_{0}^{t}C(s)ds\lesssim \max\lbrace t^{p+1},1 \rbrace,
			\end{equation*}
			\begin{equation*}
				A(t)^{p-1}\le C(t)^{p-1}\theta(t)^{p-1}\lesssim \frac{C(t)^{p-1}\cdot A(0)^{p-1}}{1-A(0)^{p-1}\max\lbrace t^{p+1},1 \rbrace}.
			\end{equation*} 
			To get the boundedness of $A(t)$, it suffices to ensure $A(0)^{p-1}\max\lbrace t^{p+1},1 \rbrace\lesssim 1$, which implies $T^{\ast}\ge K\cdot (\frac{1}{\varepsilon})^{\frac{p-1}{p+1}}$.
		\end{proof}
		\begin{rem}
			This proof is obviously applicable for other nonlinear term $F$, as long as the growth rate of $F$ can be controlled by $|x|^{p}$.
		\end{rem}
		If we turn to the defocusing case i.e $\mu=-1$, then we can derive global well-posedness theory. This result is similar with the classical theory of well-posedness of defocusing case and ill-posedness of focusing case \cite{11,18,19,20}. 
		
		\begin{proof}[Proof of Theorem \ref{Th3}]
			Based on the energy conservation in Theorem \ref{ec2}, we know that $\|\partial_{t}u(\cdot,t)\|_{\ell^{2}(\Z^{d})}$ is uniformally bounded. Then by the fundamental theorem of calculus, we obtain that  $\|u(\cdot,t)\|_{\ell^{2}(\Z^{d})}$ is bounded in every finite time. According to the continuation criterion in Theorem \ref{Th1}, we immediately get the proof of global existence of the solution for the equation (\ref{equation3}) or  (\ref{dfe1}) with $\mu=-1$.
		\end{proof} 
		\begin{rem}
			The difference between Theorem \ref{j} and Theorem \ref{Th3} comes from the conserved energy in Theorem \ref{ec2}. For defocusing case ($\mu=-1$), the energy is always positive and can control $\|\partial_{t}u(\cdot,t)\|_{\ell^{2}(\Z^{d})}$, but for focusing case ($\mu=1$), the conserved energy is not always positive and may fail to control $\|\partial_{t}u(\cdot,t)\|_{\ell^{2}(\Z^{d})}$. A simple example $u=C_{p}(t_{0}-t)^{\frac{-2}{p-1}}$, with some appropriate constant $C_{p}$, shows that there is no global $C^{2}([0,T];\ell^{\infty}(\Z^{d}))$ solution for focusing case.
		\end{rem}
		Next if we require stronger assumption on the nonlinear term, then we can also get global well-posedness theory.
		\begin{thm}
			If $g^{jk}=\delta_{jk}$ and we further require $\nabla F$ is bounded, then the equation (\ref{equation1}) has a global solution.
		\end{thm}
		\begin{proof}
			We have the following
			\begin{equation*}
				A(t)\le C(t)\left(A(0)+\int_{0}^{t}A(s)ds\right)\le C(T^{\ast})\left(A(0)+\int_{0}^{t}A(s)ds\right).
			\end{equation*}
			Then the Gronwall inequality ensures the boundedness of $A(t)$.
		\end{proof}
		\section{Some related results}
		In this section, we will present some results that partially coincide with the content of our main results, but are noteworthy in their own right.
		\par 
		We first introduce some notation and a useful lemma, by which one can derive not only the local uniqueness and existence, but also continuous dependence of initial data. 
		The nonlinear dispersive equation is defined as follows
		\begin{equation*}
			\left\{
			\begin{aligned}
				&  \partial_{t} u-Lu=N(u),   \\
				&  u(0)=u_{0}\in D,
			\end{aligned}
			\right.
		\end{equation*}
		where $u:I\subseteq\R\to D$, $D$ is a Banach space, $L$ is a linear operator, and $N$ is a nonlinear operator.  Then we have following Duhamel's formula
		\begin{equation}\label{duh}
			u(t)=e^{tL}u_{0}+\int_{0}^{t}e^{(t-s)L}Nu(s)ds\equiv u_{\mathrm{lin}}+ DN(u),
		\end{equation}
		where $u_{\mathrm{lin}}:=e^{tL}u_{0}$ and $DF:=\int_{0}^{t}e^{(t-s)L}F(s)ds$.
		\par 
		Then we state the key lemma \cite{11}, which plays the central role in Theorem \ref{b}.
		\begin{lemma}\label{TL}
			Let $B_{1},B_{2}$ be two Banach space. If we have linear operator $D:B_{1}\to B_{2}$ with bound 
			\begin{equation}\label{cond1}
				\|DF\|_{B_{2}}\le C_{0}\|F\|_{B_{1}},
			\end{equation}
			for all $F\in B_{1}$ and some constant $C_{0}>0$. Suppose that we have a nonlinear operator $N:B_{2}\to B_{1}$, with $N(0)=0$, which obeys Lipschitz bounds 
			\begin{equation}\label{cond2}
				\|N(u)-N(v)\|_{B_{1}}\le \frac{1}{2C_{0}}\|u-v\|_{B_{2}},
			\end{equation}
			for all $u,v\in \lbrace u\in B_{2}:\|u\|_{B_{2}}\le \varepsilon\rbrace$ and some $\varepsilon>0$. Then for all $u_{\mathrm{lin}}\in \lbrace u\in B_{2}:\|u\|_{B_{2}}\le \frac{\varepsilon}{2}\rbrace$ there exists a unique solution $u\in\lbrace u\in B_{2}:\|u\|_{B_{2}}\le \varepsilon\rbrace $ of the equation (\ref{duh}), with map $u_{\mathrm{lin}}\mapsto u$ being Lipschitz with Lipschitz constant $2$.
		\end{lemma}
		Next, we derive local well-posedness theory for physically important equation (\ref{dfe1}).
		
		\begin{thm}\label{b}
			For each $R>0$, there exists $T=T(d,p,R)>0$ such that for all initial data $(f,g) \in \lbrace (f,g)\in \ell^{2,k}(\Z^{d})\times \ell^{2,k}(\Z^{d}):\|f\|_{\ell^{2,k}(\Z^{d})}+\|g\|_{\ell^{2,k}(\Z^{d})}<R\rbrace$, there exists a unique classical solution $u\in C^{2}(\left[0,T\right];\ell^{2,k}(\Z^d))$ of the equation (\ref{dfe1}), $k=0,$ or $1$. Moreover, the map $(f,g) \mapsto u$ is Lipschitz continuous. 
		\end{thm}
		\begin{proof}
			We consider the following correspondence
			\[L=
			\begin{bmatrix}
				0 & 1 \\
				\Delta & 0
			\end{bmatrix}, \\ \quad e^{tL}=
			\begin{bmatrix}
				\cos(t\sqrt{-\Delta}) & \dfrac{\sin(t\sqrt{-\Delta})}{\sqrt{-\Delta}} \\
				-\sin(t\sqrt{-\Delta})\cdot \sqrt{-\Delta} & \cos(\sqrt{-\Delta})
			\end{bmatrix} , \quad u_{0}=\begin{bmatrix}
				f \\
				g
			\end{bmatrix},\]   
			
			\[\widetilde{u}(t):=
			\begin{bmatrix}
				u(t) \\
				\partial_{t}u(t) 
			\end{bmatrix}, \\ \quad  N\left(\begin{bmatrix}
				u_{1} \\
				u_{2} 
			\end{bmatrix}\right)=
			\begin{bmatrix}
				0  \\
				\mu |u_{1}|^{p-1}u_{1} 
			\end{bmatrix} ,\quad D\left(\begin{bmatrix}
				F_{1} \\
				F_{2} 
			\end{bmatrix}\right)=\int_{0}^{t}e^{(t-s)L}\begin{bmatrix}
				F_{1} \\
				F_{2} 
			\end{bmatrix}ds,\]
			where the notation $\cos(t\sqrt{-\Delta})$ is defined as follows
			\begin{equation*}
				\cos(t\sqrt{-\Delta})u:=\mathcal{F}^{-1}(\cos(t\cdot K(x))\mathcal{F}(u)),
			\end{equation*}  
			$K(x)$ is defined in Theorem \ref{energy estimate1}. Other entries in $e^{tL}$ are defined similarly.
			Then  $\widetilde{u}$ satisfies the nonlinear dispersive equation with the above correspondence. 
			\par 
			Letting $B_{1}=B_{2}=C(\left[0,T\right];\ell^{2,k}(\Z^{d}))$, we check the conditions of Lemma \ref{TL}.
			\par 
			For the condition (\ref{cond1}), we immediately derive it with the Minkowski inequality and the isomorphism between $\ell^{2,k}(\Z^{d})$ and $H^{k}(\mathbb{T}^{d})$. Moreover, the constant $C_{0}$ is $O(T)$, which implies that it can be very small in local time.
			\par 
			For the condition (\ref{cond2}),  we can use the following inequality
			\begin{equation*}
				||u|^{p-1}u-|v|^{p-1}v|\lesssim |u-v|\cdot \max\lbrace |u|^{p-1}, |v|^{p-1}\rbrace.
			\end{equation*}
			Then we can apply Lemma \ref{TL} and prove the theorem. 
		\end{proof}
		\begin{rem}
			The local existence and uniqueness part is contained in the previous Theorem \ref{Th1}. However, Theorem \ref{b} shows the continuous (in fact, Lipschitz) dependence of initial data, which will be useful in approximation arguments.
		\end{rem}
		Another interesting property of the solution for the equation (\ref{dfe1}) is that it can persist in strong norm with weaker assumption.
		\begin{thm}
			If $I\ni 0$ is a time interval,  $u\in C^{2}([0,T];\ell^{2}(\Z^{d}))\bigcap L^{p-1}([0,T];\ell^{\infty}(\Z^{d}))$ is a solution to the equation (\ref{dfe1}), and $f,g\in \ell^{2,k}(\Z^{d}),$ then for $k=0,1$, we have the estimate
			\begin{equation*}
				\|u(\cdot,t)\|_{\ell^{2,k}(\Z^{d})}+\|\partial_{t}u(\cdot,t)\|_{\ell^{2,k}(\Z^{d})}\le C(1+|T|^{k+1}) (\|f\|_{\ell^{2,k}(\Z^{d})}+\|g\|_{\ell^{2,k}(\Z^{d})})exp(C'\|u\|_{L^{p-1}\ell^{\infty}(\Z^{d}\times [0,T])}).
			\end{equation*} 
		\end{thm}
		\begin{proof}
			This is a direct consequence of Theorem \ref{energy estimate1}, combined with the inequality $\||u|^{p-1}u\|_{\ell^{2,k}(\Z^{d})}\le \|u\|_{\ell^{2,k}(\Z^{d})}\cdot \|u\|_{\ell^{\infty}(\Z^{d})}^{p-1}$ and the Gronwall inequality.
		\end{proof}
		\begin{rem}
			This result shows that the weaker norm $\ell^{\infty}(\Z^{d})$ can ensure the stronger norm $ \ell^{2,k}(\Z^{d})$. Besides, this theorem is actually a stronger version of continuation criterion. Recall that in Theorem \ref{Th1}, we require $\|u(\cdot,t)\|_{\ell^{\infty}(\Z^{d})}+\|\partial_{t}u(\cdot,t)\|_{\ell^{\infty}(\Z^{d})}$ is bounded in $\left[0,T^{\ast}\right)$, which is a stronger assumption than $u\in L^{p-1}(\left[0,T^{\ast}\right);\ell^{\infty}(\Z^{d}))$.
		\end{rem}
		\par 
		Next, we will use explicitly time-dependent energy estimate to derive the local well-posedness theory for discrete  wave equation with quadratic derivatives as the nonlinear term defined as 
		\begin{equation}\label{fad}
			\left\{
			\begin{aligned}
				&  \partial_{t}^{2} u-\Delta u=|\partial_{t}u|^{2} \; or\; |\partial_{j}u|^{2},   \\
				&  u(x,0)=f(x),  \partial_{t}u(x,0)=g(x), (x,t)\in \Z^{d} \times \R.
			\end{aligned}
			\right.
		\end{equation}  The proof is inspired by \cite{21}.
		\begin{thm}
			If $f,g\in \ell^{2,k}(\Z^{d}), k=0,1 $, then there exists $T>0$, such that the equation (\ref{fad}) has a unique classical solution $u\in C^{2}(\left[0,T\right];\ell^{2,k}(\Z^{d})).$
		\end{thm}
		\begin{proof}
			We only prove the case $|\partial_{t}u|^{2},$ since the proof of $|\partial_{j}u|^{2}$ is similar. At first, we define
			\begin{equation*}
				\Gamma:=C^{1}(\left[0,T\right];\ell^{2,k}(\Z^{d})); \;  \|u\|_{\Gamma}:=\sup_{0\le t\le T}\left[\|u(\cdot,t)\|_{\ell^{2,k}(\Z^{d})}+ \|\partial_{t}u(\cdot,t)\|_{\ell^{2,k}(\Z^{d})}  \right]. 
			\end{equation*}
			\par To find a solution, we define the map $\Phi(v)=u$  be the solution of 
			\begin{equation*}
				\left\{
				\begin{aligned}
					&  \partial_{t}^{2} u-\Delta u=|\partial_{t}v|^{2} \; or\; |\partial_{j}v|^{2},   \\
					&  u(x,0)=f(x),  \partial_{t}u(x,0)=g(x), (x,t)\in \Z^{d} \times \R.
				\end{aligned}
				\right.
			\end{equation*}
			To apply contraction mapping theorem, we consider $X:=\lbrace v\in \Gamma; \|v\|_{\Gamma}\le A \rbrace$, where A is to be determined later. Then we only need to show that $(1):\Phi(X)\subseteq X$ and $(2):\Phi$ is a contraction mapping in $X$.
			\par 
			For (1), we use explicitly time-dependent energy estimate in Theorem \ref{energy estimate1} and get
			\begin{equation*}
				\forall v\in X, \; \|\Phi(v)\|_{\Gamma}\lesssim (1+T^{k+1})\left[\|f\|_{\ell^{2,k}(\Z^{d})}+\|g\|_{\ell^{2,k}(\Z^{d})}+\int_{0}^{T}\||\partial_{t}u|^{2}(\cdot,t)\|_{\ell^{2,k}(\Z^{d})}dt\right].
			\end{equation*}
			Noting that $\||\partial_{t}u|^{2}(\cdot,t)\|_{\ell^{2,k}(\Z^{d})}\le A^{2}$, we see that $\|\Phi(v)\|_{\Gamma}\le C(1+T^{k+1})(R+TA^{2}),$ where $C>0$ is a constant and  $R:=\|f\|_{\ell^{2,k}(\Z^{d})}+\|g\|_{\ell^{2,k}(\Z^{d})}$. Then choosing $T\ll_{A} 1 $ and $R\ll A$, we can ensure $C(1+T^{k+1})(R+TA^{2})\le A$, which means $\Phi(X)\subseteq X$.
			\par 
			For (2), noting that $$\Box\left[\Phi(v_{2})-\Phi(v_{1})\right](x,t)=|\partial_{t}v_{2}|^{2}(x,t)-|\partial_{t}v_{1}|^{2}(x,t), \ \forall v_{1},v_{2}\in X,$$ we apply the energy estimate again and get
			\begin{equation*}
				\|\Phi(v_{2})-\Phi(v_{1})\|_{\Gamma}\lesssim (1+T^{k+1})\int_{0}^{T}(\|\partial_{t}v_{1}(\cdot,t)\|_{\ell^{2,k}(\Z^{d})}+\|\partial_{t}v_{2}(\cdot,t)\|_{\ell^{2,k}(\Z^{d})})\cdot \|(\partial_{t}v_{2}-\partial_{t}v_{1})(\cdot,t)\|_{\ell^{2,k}(\Z^{d})}dt
			\end{equation*}
			\begin{equation*}
				\lesssim (1+T^{k+1})T\cdot (\|v_{1}\|_{\Gamma}+\|v_{2}\|_{\Gamma})\cdot \|v_{2}-v_{1}\|_{\Gamma}.
			\end{equation*}
			Therefore, we have $$\|\Phi(v_{2})-\Phi(v_{1})\|_{\Gamma}\le CAT(1+T^{k+1})\|v_{2}-v_{1}\|_{\Gamma},$$ for some constant $C>0$. Then $CAT(1+T^{k+1})<1 $, as $T\ll_{A}1$, which ensures $\Phi$ is a contraction mapping in $X$. 
		\end{proof}
		\begin{rem}
			Although this result is contained in previous Theorem \ref{Th1}, this provides a simple proof, which directly shows the strength of the energy estimate.
			Besides, this proof is applicable for other nonlinear terms like $|u|^{p}, |\partial_{t}u|^{p},\cdots $, but quadratic derivative nonlinear term has deep connection with Faddeev equation, which is still not be solved completely. 
		\end{rem}
		In the end, we present an energy estimate, which has no time-dependent coefficient and exponential term.
		
		\begin{thm}\label{www}
			When $g^{jk}\equiv \delta_{jk}$, which means $\Box_{g}$ is the traditional d'Alembert operator $\Box:=\partial_{t}^{2}-\Delta$, we have the following stronger energy estimate for $u\in C^{2}([0,T];\ell^{2}(\Z^{d}))$
			\begin{equation*}
				\|u'(\cdot,t)\|_{\ell^{2}(\Z^{d})}\le \left(\| u'(\cdot,0)\|_{\ell^{2}(\Z^{d})}+\int_{0}^{t}\|\Box u(\cdot,s)\|_{\ell^{2}(\Z^{d})}ds\right).
			\end{equation*}
		\end{thm}
		\begin{proof}
			We define the energy $E(t)$ as 
			\begin{equation*}
				E(t):=\sum_{m\in \Z^{d}}\left(|\partial_{t}u(m,t)|^{2}+\sum_{j=1}^{d}|\partial_{j}u(m,t)|^2\right).
			\end{equation*}
			Differentiating $E(t)$, we derive the following identity
			\begin{equation*}
				\frac{d}{dt}E(t)=\sum_{m\in \Z^{d}}\partial_{t}^{2}u(m,t)\overline{\partial_{t}u(m,t)}+\partial_{t}u(m,t)\overline{\partial_{t}^{2}u(m,t)}+(\ast)
			\end{equation*}
			\begin{equation*}
				(\ast)= \sum_{j=1}^{d}\sum_{m\in \Z^{d}}\partial_{j}(\partial_{t}u)(m,t)\overline{\partial_{j}u(m,t)}+\overline{\partial_{j}(\partial_{t}u)(m,t)}\partial_{j}u(m,t).
			\end{equation*}
			Applying integration by part formula from Theorem \ref{integration by part}, we can derive
			\begin{equation*}
				(\ast)=-\sum_{j=1}^{d}\sum_{m\in \Z^{d}}\partial_{t}u(m,t)\overline{\partial_{jj}u(m,t)}+\partial_{jj}u(m,t)\overline{\partial_{t}u(m,t)}
			\end{equation*}
			\begin{equation*}
				=-\sum_{m\in \Z^{d}}\partial_{t}u(m,t)\overline{\Delta u(m,t)}+\Delta u(m,t)\overline{\partial_{t}u(m,t)}.
			\end{equation*}
			Substituting the d'Alembert operator $\Box=\partial_{t}^{2}-\Delta$, we have the following identity
			\begin{equation*}
				\frac{d}{dt}E(t)=\sum_{m\in \Z^{d}}\Box u(m)\overline{\partial_{t}u(m)}+\partial_{t}u(m)\overline{\Box u(m)}\le 2 E(t)^{1/2}\|\Box u\|_{\ell^{2}(\Z^{d})}.
			\end{equation*}
			Then dividing by $2E(t)^{1/2}$ and applying the fundamental theorem of calculus, we immediately get the desired result.
		\end{proof}
		\begin{rem}
			This energy estimate in Theorem \ref{www} is strong, since it's of the form $\le$ instead of $\lesssim,$ and it has no exponential term. Moreover, the explicitly time-dependent energy estimate in Theorem \ref{energy estimate1} is a direct consequence of this energy estimate, by applying the fundamental theorem of calculus.
		\end{rem}
		
		\newpage
		\section*{Acknowledgement}
		B. Hua is supported by NSFC, No. 12371056, and by Shanghai Science
		and Technology Program [Project No. 22JC1400100]. J. Wang is supported by NSFC, No. 123B1035.
		
		\bigskip
		\bigskip
		

		\bibliographystyle{alpha}
		\bibliography{NNonlinearwave}

	\end{document}